\documentclass[12pt,letterpaper]{article}

\usepackage[margin=2.75cm]{geometry}
\usepackage{amsmath,amsfonts,amssymb,amsthm}
\usepackage{graphicx}
\usepackage{subfig}
\usepackage[font=scriptsize,labelfont=bf,labelsep=space]{caption}
\usepackage{placeins}
\usepackage{float}
\usepackage{pgfplots}
\usepackage{enumerate}
\usepackage[utf8]{inputenc}
\usepackage[T1]{fontenc}
\usepackage{lmodern}
\usepackage[affil-it]{authblk}
\usepackage{setspace}
\newcommand{\N}{\mathbb N}

\newcommand{\R}{\mathbb R}

\newcommand{\QQ}{\mathcal Q}
\newcommand{\RR}{\mathcal R}
\newcommand{\PP}{\mathcal P}
\newcommand{\II}{\mathcal I}
\newcommand{\supp}{\text{supp}}

\newcommand{\xxx}{\mathbf{x}}
\newcommand{\yyy}{\mathbf{y}}
\newcommand{\ttt}{\mathbf{t}}
\newcommand{\ddd}{\mathbf{d}}

\newtheorem{defn}{Definition}[section]
\newtheorem{theorem}{Theorem}[section]
\newtheorem{lemma}{Lemma}[section]

\makeatletter
\newcommand{\pushright}[1]{\ifmeasuring@#1\else\omit\hfill$\displaystyle#1$\fi\ignorespaces}
\newcommand{\pushleft}[1]{\ifmeasuring@#1\else\omit$\displaystyle#1$\hfill\fi\ignorespaces}
\makeatother

\makeatletter
\def\@maketitle{%
  \newpage
  \null
  \vskip 2em%
  \begin{center}%
  \let \footnote \thanks
    {\Large\bfseries \@title \par}%
    \vskip 1.5em%
    {\normalsize
      \lineskip .5em%
      \begin{tabular}[t]{c}%
        \@author
      \end{tabular}\par}%
    \vskip 1em%
    {\normalsize \@date}%
  \end{center}%
  \par
  \vskip 1.5em}
\makeatother

\numberwithin{equation}{section}
\numberwithin{figure}{section}
\allowdisplaybreaks[2]

\begin{document}

\title{Real-time, local spline interpolation schemes \\on bounded intervals}
\author{Maria D. van der Walt%
\thanks{Electronic address: \texttt{maria.vanderwalt@vanderbilt.edu}}}
\affil{Department of Mathematics, Vanderbilt University, Nashville, TN 37240}
\date{}
\maketitle

\begin{abstract}
We develop a local polynomial spline interpolation scheme for arbitrary spline order on bounded intervals. Our method's local formulation, effective boundary considerations and optimal interpolation error rate make it particularly useful for real-time implementation in real-world applications.
\end{abstract}

\section{Introduction}
\label{sec_intro}

The objective of this paper is to explicitly construct polynomial spline interpolation operators for real-valued functions defined on bounded intervals. Specifically, for a function $f: \ [a,b]\rightarrow\R$ and a strictly increasing sequence of sampling points $\left\{ y_0, y_1,\ldots,y_N \right\} \subseteq [a,b]$, our goal is to construct an interpolation operator ${\PP}_m$ in terms of the $m^{\textup{th}}$ order B-splines (where $m\geq3$) such that the following conditions are satisfied:
\begin{enumerate}[(i)]
\item \label{it_blending_cond_i}
${\PP}_m$ is local in the sense that the value of ${\PP}_{m}f$ at any $y^{*}\in[a,b]$ only depends on the values of $f$ in a small neighborhood of $y^{*}$;
\item \label{it_blending_cond_ii}
${\PP}_m$ preserves polynomials of degree $\leq m-1$; that is,
\begin{equation}
({\PP}_{m} p)(x) = p(x), \quad p\in\pi_{m-1}, \ x\in[a, b];
\label{eq_poly_preserve}
\end{equation}
\item \label{it_blending_cond_iii}
${\PP}_{m}f$ interpolates $f$ at the interpolation points $\left\{ y_0, y_1,\ldots,y_N \right\}$; that is,
\begin{equation}
({\PP}_{m} f)(y_{i}) = f(y_{i}), \quad i=0,1,\ldots,N;
\label{eq_interpolatory_cond}
\end{equation}
\item \label{it_blending_cond_iv}
${\PP}_{m}$ preserves derivatives of $f$ of order $\ell$ (for $\ell=1,2,\ldots,m-1$) at $a$ and $b$, so that
\begin{equation}
({\PP}_{m} f)^{(\ell)}(a) = f^{(\ell)}(a); \quad ({\PP}_{m} f)^{(\ell)}(b) = f^{(\ell)}(b), \quad \ell=1,2,\ldots,m-1.
\label{eq_hermite_interpolatory_cond}
\end{equation}
\end{enumerate}

(In (\ref{it_blending_cond_iv}), since the derivatives of the function $f$ might not be known in practice, we approximate the $\ell^{\textup{th}}$ derivative of $f$ at $a$ and $b$ by the $\ell^{\textup{th}}$ order divided difference of $f$ at $a$ and $b$, respectively, when applying our method.)\\

Our idea is to start by developing a locally supported quasi-interpolation operator ${\QQ}_{m}$ in terms of the $m^{\textup{th}}$ order B-splines on the bounded interval $[a,b]$ to achieve properties (\ref{it_blending_cond_i}) and (\ref{it_blending_cond_ii}), as well as high approximation order \cite{chuistudents,lyche1975local}. Then, to satisfy the Hermite interpolation conditions (\ref{it_blending_cond_iii}) and (\ref{it_blending_cond_iv}) above (while preserving local support in (\ref{it_blending_cond_i})), we construct a local interpolation operator ${\RR}_{m}$ as well, leading to the blending operator ${\PP}_{m}$, as introduced in \cite{chuidiamond}, defined by
\begin{equation}
{\PP}_{m} := {\RR}_{m} \oplus {\QQ}_{m},
\label{eq_PP}
\end{equation}
where
\begin{equation}
{\RR}_{m} \oplus {\QQ}_{m} := {\QQ}_{m} + {\RR}_{m}({\II}_{m} - {\QQ}_{m}) = {\QQ}_{m} + {\RR}_{m} - {\RR}_{m}{\QQ}_{m},
\label{eq_RR+QQ}
\end{equation}
with ${\II}_{m}$ denoting the identity operator. With this definition, ${\PP}_m$ achieves the four objectives in (\ref{it_blending_cond_i})-(\ref{it_blending_cond_iv}) above (as we will show later in this paper).\\

Our approach has several advantages. First, the local formulation make this spline interpolation scheme particularly useful, since it facilitates real-time implementation for fast computation (without matrix inversions). Second, our method provides an effective way to take care of interpolation at the endpoints of the interval $[a,b]$, without extending the function $f$ in any way. Moreover, we will show that this interpolation operator attains the optimal interpolation error rate (away from the boundaries) compared to the traditional spline interpolation scheme. These properties make our method suitable for a wide range of applications, including the popular empirical mode decomposition (EMD) scheme \cite{Huang08031998} for non-stationary signal analysis, where an interpolation scheme is required in its sifting procedure (see \cite{A} for an example of such an application).\\

In Section \ref{sec_quasi}, we start by developing a quasi-interpolation operator ${\QQ}_m$, in terms of the $m^{\textup{th}}$ order B-splines, that preserves polynomials $p$ of degree $\leq m-1$. We remark that local spline quasi-interpolants were first constructed by De Boor and Fix in \cite{deboorfix}, by using both functional and derivative values of the function $f$. Later, Lyche and Schumaker in \cite{lyche1975local} considered local spline quasi-interpolants written in terms of divided differences of $f$. In contrast, our quasi-interpolation operator mainly makes use of functional values, employing divided differences of $f$ only near the endpoints $x=a$ and $x=b$. This is advantageous from a numerical point of view. Our scheme is based on a quasi-interpolation scheme for real-time application described in \cite{chuistudents}; however, the method in \cite{chuistudents} is derived for data values on an unbounded interval, and is adapted here for a bounded interval. Next, in Section \ref{sec_local}, we develop a local interpolation operator ${\RR}_m$ that will achieve properties (\ref{it_blending_cond_iii}) and (\ref{it_blending_cond_iv}), inspired by a method in \cite{2014arXiv1406.1276C}. However, we note that, in contrast to the scheme in \cite{2014arXiv1406.1276C}, our method is defined for a bounded interval and accommodates the Hermite interpolation conditions in (\ref{it_blending_cond_iv}). In Section \ref{sec_blending}, we combine the quasi-interpolation operator ${\QQ}_m$ with the local interpolation operator ${\RR}_m$ in the blending operator, defined in (\ref{eq_PP})-(\ref{eq_RR+QQ}), and show that ${\PP}_m$ preserves the properties of both its constituent operators. In Section \ref{sec_blending_approx_order}, we investigate the error bounds for the blending interpolation operator applied to a general real-valued function $f$. Final remarks follow in Section \ref{sec_final}.


\section{Quasi-interpolation}
\label{sec_quasi}

For a given sequence
\begin{equation}
{\xxx}: \quad x_{-m+1} = \cdots = a = x_0 < x_{1} < \cdots < x_{N+1} = b = \cdots = x_{N+m},
\label{eq_xxx}
\end{equation}
where $m\geq3$ and $N$ is a positive integer, let $S_{{\xxx},m}[a,b]$ denote the linear space of $m^{\textup{th}}$ order polynomial splines on $[a,b]$ with knots in ${\xxx}$. As discussed in \cite{curry_schoenberg,de2001practical}, a locally supported basis for $S_{{\xxx},m}[a,b]$ is given by the set of normalized $m^{\textup{th}}$ order B-splines $\left\{ N_{{\xxx},m,j}: \ j=-m+1,\ldots,N \right\},$ where each $N_{{\xxx},m,j}$ is defined in terms of divided differences of truncated powers:\begin{equation}
N_{{\xxx},m,j}(x) := (x_{j+m} - x_{j}) [x_{j},\ldots,x_{j+m}] (\cdot - x)^{m-1}_{+}, \quad j=-m+1,\ldots,N.
\label{eq_b_spline}
\end{equation}
Let $f$ be a real-valued function, and let
\begin{equation}
{\yyy}: \quad a = y_0 < y_{1} < \cdots < y_{N} = b
\label{eq_yyy_knots_half}
\end{equation}
be a given sequence of (non-uniform) sampling points. In what follows, we develop a spline quasi-interpolation operator ${\QQ}_m$, under the assumption that the B-spline knots are chosen to lie midway between consecutive sampling points -- more precisely, the knots $x_0,\ldots,x_{N+1}$ in ${\xxx}$ in (\ref{eq_xxx}) are defined by
\begin{equation}
x_0 := y_0; \quad x_i := \tfrac{1}{2}\left( y_{i-1} + y_{i}\right), \ i=1,\ldots,N; \quad x_{N+1} := y_N.
\label{eq_xxx_half}
\end{equation}
We remark that the spline knots ${\xxx}$ may also be chosen to coincide with the sampling points ${\yyy}$, with slight alterations of the formulations below. Details of this setup, in terms of the cubic B-splines, may be found in \cite{A}.\\

We will need the following notations. First, $D(x_{j},\ldots,x_{j+m-1})$ denotes the Vandermonde determinant of $x_{j},\ldots,x_{j+m-1}$; that is,
\begin{equation}
D(x_{j},\ldots,x_{j+m-1}) =
\left| \begin{array}{cccc}
1 & 1 & \cdots & 1 \\
x_{j} & x_{j+1} & \cdots & x_{j+m-1} \\
\vdots & \vdots & & \vdots \\
x_{j}^{m-1} & x_{j+1}^{m-1} & \cdots & x_{j+m-1}^{m-1} \end{array} \right|;
\label{eq_vandermonde}
\end{equation}
and $D(x_{j},\ldots, x_{j+k-1}, \xi_{{\xxx},m,\ell}, x_{j+k+1}, \ldots, x_{j+m-1})$ is obtained from $D(x_{j},\ldots,x_{j+m-1})$ by replacing its $(k+1)^{\textup{th}}$ column with the vector
\begin{equation}
\xi_{{\xxx},m,\ell} := [\xi^{0}_{{\xxx}}(\ell),\ldots,\xi^{m-1}_{{\xxx}}(\ell)]^{T},
\label{eq_xi1}
\end{equation}
with
\begin{equation}
\left\{ \begin{aligned}
\xi^{0}_{{\xxx}}(\ell) &= 1;\\
\xi^{n}_{{\xxx}}(\ell)&=\frac{\sigma^{n}(x_{\ell+1},\ldots,x_{\ell+m-1})}{\binom{m-1}{n}}, \quad n=1,\ldots,m-1,
\end{aligned}
\right.
\label{eq_xi2}
\end{equation}
and where $\sigma^{n}(x_{\ell+1},\ldots,x_{\ell+m-1})$ denotes the classical symmetric function, defined by
\begin{equation}
\left\{ \begin{aligned}
\sigma^{0}(x_{\ell+1},\ldots,x_{\ell+m-1}) &= 1;\\
\sigma^{n}(x_{\ell+1},\ldots,x_{\ell+m-1})&=\sum_{\ell+1\leq t_{1} < t_{2} < \cdots < t_{n} \leq \ell+m-1} x_{t_{1}}x_{t_{2}}\cdots x_{t_{n}}, \quad n=1,\ldots,m-1,\\
\end{aligned}
\right.
\label{eq_sigma}
\end{equation}
with the definition that $\sigma^{n}(x_{\ell+1},\ldots,x_{\ell+m-1}) = 0$ if $n\geq m$. Furthermore, \\$D_{C}(x_j, \ldots,x_{j+p},x_{j+p}^{(1)}, \ldots, x_{j+p}^{(q)}, x_{j+p+1},\ldots,x_{j+m-q-1})$ denotes the confluent Vandermonde determinant; that is, for $\ell=1,\ldots,q$, the $(p+1+\ell)^{\textup{th}}$ column of $D_{C}$ is given by
\begin{equation}
(D_{C})_{k,p+1+\ell} = \begin{cases}
0, & {\textup{if}} \ k\leq \ell; \\
\frac{(k-1)!}{(k-1-\ell)!}x_{j+p}^{k-1-\ell}, & {\textup{if}} \ k>\ell.
\end{cases}
\label{eq_confl_vandermonde}
\end{equation}
(In other words, confluent columns are derivatives of the original Vandermonde columns.) The remaining $m-q$ columns of $D_{C}$ are regular Vandermonde columns corresponding to $x_{j},\ldots,x_{j+m-q-1}$ (as in (\ref{eq_vandermonde})). Similar as above, \\$D_{C}(x_j, \ldots,x_{j+p},x_{j+p}^{(1)}, \ldots, x_{j+p}^{(k-1)}, \xi_{{\xxx},m,\ell}, x_{j+p}^{(k+1)}, \ldots, x_{j+p}^{(q)}, x_{j+p+1},\ldots,x_{j+m-q-1})$ is obtained from $D_{C}(x_j, \ldots,x_{j+p},x_{j+p}^{(1)}, \ldots, x_{j+p}^{(q)}, x_{j+p+1},\ldots,x_{j+m-q-1})$ by replacing its $(p+k+1)^{\textup{th}}$ column with $\xi_{{\xxx},m,\ell}$.

\begin{defn}
\label{defn_quasi-interpolant_knots_half}
The quasi-interpolation operator ${\QQ}_{m}$ is defined by
\begin{equation}
({\QQ}_{m}f)(x) := \sum_{\ell=1}^{m-1} f^{(\ell)}(a)M_{m,-\ell}(x) + \sum_{i=0}^{N} f(y_{i})M_{m,i}(x) + \sum_{r=1}^{m-1} f^{(r)}(b) M_{m,N+r}(x),
\label{eq_QQ_knots_half}
\end{equation}
in terms of the spline molecules
\begin{equation}
\left\{ \begin{aligned}
& M_{m,-\ell}(x) := \sum_{j=0}^{m-1-\ell} a_{m,-\ell,j} N_{{\xxx},m,j-m+1}(x), \quad \ell=1,\ldots,m-1; \\
& M_{m,i}(x) := \sum_{j=0}^{m-1} a_{m,i,j} N_{{\xxx}, m, i+j-m+1}(x), \quad i=0,\ldots,N; \\
& M_{m,N+r}(x) := \sum_{j=r}^{m-1} a_{m,N+r,j} N_{{\xxx},m,N+j-m+1}(x), \quad r=1,\ldots,m-1,
\end{aligned}
\right.
\label{eq_M_defn_knots_half}
\end{equation}
where the coefficients are given by:
\begin{itemize}
\item For $i=0,\ldots,m-2, \ j=m-1-i,\ldots,m-1,$ and $i=m-1,\ldots,N+1-m, \ j=0,\ldots,m-1,$ and $i=N-m+2,\ldots,N, \ j=0,\ldots,N-i$:
\begin{equation}
a_{m,i,j} = \frac{D(y_{i+j-m+1}, \ldots, y_{i-1}, \xi_{{\xxx},m,i+j-m+1}, y_{i+1},\ldots,y_{i+j})}{D(y_{i+j-m+1}, \ldots, y_{i+j})};
\label{eq_a_int_half}
\end{equation}
\item For $i=0,\ldots,m-2, \ j=0,\ldots,m-2-i$:
\begin{equation}
a_{m,i,j} = \frac{D_{C}(y_{0}, y_0^{(1)}, \ldots, y_{0}^{(m-1-i-j)}, y_1, \ldots, y_{i-1}, \xi_{{\xxx},m,i+j-m+1}, y_{i+1},\ldots,y_{i+j})}{D_{C}(y_{0}, y_0^{(1)}, \ldots, y_{0}^{(m-1-i-j)}, y_1, \ldots, y_{i+j})};
\label{eq_a_lb1_half}
\end{equation}
\item For $\ell=1,\ldots,m-1, \ j=0,\ldots,m-1-\ell$:
\begin{equation}
a_{m,-\ell,j} = \frac{D_{C}(y_{0}, y_0^{(1)}, \ldots, y_{0}^{(\ell-1)}, \xi_{{\xxx},m,j-m+1}, y_{0}^{(\ell+1)}, \ldots, y_{0}^{(m-1-j)}, y_{1},\ldots,y_{j})}{D_{C}(y_{0}, y_0^{(1)}, \ldots, y_{0}^{(m-1-j)}, y_1, \ldots, y_{j})};
\label{eq_a_lb2_half}
\end{equation}
\item For $i=N-m+2,\ldots,N, \ j=N-i+1,\ldots,m-1$:
\begin{multline}
a_{m,i,j}\\
= \frac{D_{C}(y_{i+j-m+1}, \ldots, y_{i-1}, \xi_{{\xxx},m,i+j-m+1}, y_{i+1},\ldots,y_{N},y_{N}^{(1)},\ldots,y_{N}^{(i+j-N)})}{D_{C}(y_{i+j-m+1}, \ldots, y_{N},y_{N}^{(1)},\ldots,y_{N}^{(i+j-N)})};
\label{eq_a_rb1_half}
\end{multline}
\item For $r=1,\ldots,m-1, \ j=r,\ldots,m-1$:
\begin{multline}
a_{m,N+r,j} \\= \frac{D_{C}(y_{N+j-m+1}, \ldots, y_{N}, y_{N}^{(1)},\ldots, y_{N}^{(r-1)}, \xi_{{\xxx},m,N+j-m+1}, y_{N}^{(r+1)}, \ldots, y_{N}^{(j)})}{D_{C}(y_{N+j-m+1}, \ldots, y_{N}, y_{N}^{(1)},\ldots, y_{N}^{(j)})}.
\label{eq_a_rb2_half}
\end{multline}
\end{itemize}
\end{defn}

Since the molecules in the definition above are written as a linear combination of B-splines, they are compactly supported, with
\begin{equation}
\left\{ \begin{aligned}
& {\supp}M_{m,-\ell} = [x_{0}, x_{m-\ell}], \quad \ell=1,\ldots,m-1;\\
& {\supp}M_{m,i} = [x_{i-m+1}, x_{i+m}], \quad i=0,\ldots,N;\\
& {\supp}M_{m,N+r} = [x_{N-m+1+r}, x_{N+1}], \quad r=1,\ldots,m-1.
\end{aligned}
\right.
\label{eq_M_support_knots_half}
\end{equation}
This ensures the local formulation of the quasi-interpolation operator ${\QQ}_m$.\\

With these definitions, we can show that ${\QQ}_m$ preserves polynomials of degree $\leq m-1$. The following theorem originally appeared in \cite[Theorem~2.1]{chuistudents}, where the result was proved only for an unbounded interval. It is extended here to include the boundary considerations in Definition \ref{defn_quasi-interpolant_knots_half}.

\begin{theorem}
\label{thm_quasi_knots_half}
For $N\geq3m-3$, the quasi-interpolation operator ${\QQ}_m$, formulated in (\ref{eq_QQ_knots_half}) in Definition \ref{defn_quasi-interpolant_knots_half}, satisfies the condition
\begin{equation}
({\QQ}_{m}p)(x) = p(x),
\label{eq_quasi_knots_half}
\end{equation}
for all $x\in [a,b]$ and $p\in\pi_{m-1}$.
\end{theorem}

\begin{proof}

First, if $x\in[x_{2m-2}, x_{N-2m+3}]$, (\ref{eq_a_int_half}) in Definition \ref{defn_quasi-interpolant_knots_half} follows as in \cite[Theorem~2.1]{chuistudents}. Next, let $x\in[a,x_{2m-2}]$, so that (\ref{eq_QQ_knots_half}) becomes
\begin{equation}
({\QQ}_{m}f)(x) = \sum_{\ell=1}^{m-1}f^{(\ell)}(a)M_{m,-\ell}(x) + \sum_{i=0}^{3m-4} f(y_{i})M_{m,i}(x),
\label{eq_Qm_lb}
\end{equation}
from the support properties of $M_{m,i}$ in (\ref{eq_M_support_knots_half}). We proceed to show that the constants $a_{m,i,j}, \ i=-m+1,\ldots,m-2,$ satisfy the formulation (\ref{eq_a_int_half})-(\ref{eq_a_lb2_half}) in Definition \ref{defn_quasi-interpolant_knots_half} if ${\QQ}_m$ satisfies (\ref{eq_quasi_knots_half}), for $p(x) = x^{t}, \ t = 0,\ldots,m-1$. By using (\ref{eq_Qm_lb}) and the first two equations in (\ref{eq_M_defn_knots_half}), the left hand side of (\ref{eq_quasi_knots_half}) becomes
\begin{align}
({\QQ}_{m}p)(x) &= \sum_{\ell=1}^{m-1} \tbinom{t}{\ell}\ell!y_0^{t-\ell} \sum_{j=0}^{m-1-\ell} a_{m,-\ell,j} N_{{\xxx},m,j-m+1}(x) {+ \sum_{i=0}^{3m-4} y_{i}^{t} \sum_{j=0}^{m-1} a_{m,i,j} N_{{\xxx},m,i+j-m+1}(x)} \nonumber \\
&= \sum_{\ell=1}^{m-1}\sum_{k=-m+1}^{-\ell} \tbinom{t}{\ell}\ell!y_0^{t-\ell} a_{m,-\ell,k+m-1}N_{{\xxx},m,k}(x) \nonumber \\
& \pushright{+ \sum_{j=0}^{m-1}\sum_{k=j-m+1}^{2m+j-3} y_{k-j+m-1}^{t}a_{m,k-j+m-1,j} N_{{\xxx},m,k}(x)} \nonumber \\
&= \sum_{k=-m+1}^{-1} \sum_{\ell=1}^{-k} \tbinom{t}{\ell}\ell!y_0^{t-\ell} a_{m,-\ell,k+m-1}N_{{\xxx},m,k}(x) \nonumber \\
& \quad \quad {+ \sum_{k=-m+1}^{3m-4}\sum_{j=\max\left\{0,k-2m+3\right\}}^{\min\left\{m-1,k+m-1\right\}} y_{k-j+m-1}^{t}a_{m,k-j+m-1,j} N_{{\xxx},m,k}(x).}
\label{eq_proof2}
\end{align}
Combining (\ref{eq_quasi_knots_half}) and (\ref{eq_proof2}) with Marsden's identity on the interval $[a,x_{2m-2}]$, given by
\begin{equation*}
x^{t} = \sum_{k=-m+1}^{2m-3} \xi^{t}_{\xxx}(k) N_{{\xxx},m,k}(x), \quad t=0,\ldots,m-1,
\end{equation*}
we have
\begin{multline*}
\sum_{k=-m+1}^{-1} \left[\sum_{\ell=1}^{-k} \tbinom{t}{\ell}\ell!y_0^{t-\ell} a_{m,-\ell,k+m-1} + \sum_{j=0}^{k+m-1} y_{k-j+m-1}^{t}a_{m,k-j+m-1,j}\right] N_{{\xxx},m,k}(x) \\
+ \sum_{k=0}^{m-2} \left[\sum_{j=0}^{m-1} y_{k-j+m-1}^{t} a_{m,k-j+m-1,j}\right] N_{{\xxx},m,k}(x)
= \sum_{k=-m+1}^{m-2} \xi^{t}_{\xxx}(k)N_{{\xxx},m,k}(x).
\end{multline*}
The result follows by comparing the left hand side and right hand side for $k=-m+1,\ldots,m-2,$ and using Cramer's rule. The formulation (\ref{eq_a_int_half}) corresponds to $k=0,\ldots,m-2$, while (\ref{eq_a_lb1_half}) and (\ref{eq_a_lb2_half}) (in terms of confluent Vandermonde determinants) correspond to $k=-m+1,\ldots,-1$.\\

Lastly, when $x\in [x_{N+3-2m},b]$, we obtain the coefficients in (\ref{eq_a_int_half}), (\ref{eq_a_rb1_half}) and (\ref{eq_a_rb2_half}) under the requirement that ${\QQ}_m$ satisfies (\ref{eq_quasi_knots_half}). The proof follows a similar pattern as the proof in part (b) above and is omitted here.
\end{proof}


\section{Local interpolation}
\label{sec_local}

We proceed to define the local spline interpolation operator ${\RR}_m$ to satisfy the Hermite interpolation conditions (\ref{it_blending_cond_iii}) and (\ref{it_blending_cond_iv}) in Section \ref{sec_intro}. As in (\ref{eq_xxx_half}), we construct the spline knots to lie midway between consecutive sampling points. We also consider a knot sequence ${\ttt} \supset {\yyy}$, which is constructed as follows:\\

First, suppose $m$ is even. With $q_m := m/2$, we insert $q_m-1$ equally spaced knots in between every two (interior) knots of ${\yyy}$, so that $t_{m+(j-1)q_m} = y_j$ for $j=1,\ldots,N-1.$ Furthermore, to facilitate the Hermite interpolation conditions at the boundaries, we also insert $m-1$ evenly spaced knots $t_1,\ldots,t_{m-1}$ in the interval $\left(y_0,y_1\right)$, with $t_0 := x_0$, as well as $m-1$ evenly spaced knots $t_{m+(N-2)q_m+1},\ldots,t_{2m+(N-2)q_m-1}$ in the interval $\left(y_{N-1},y_{N}\right)$, with $t_{2m+(N-2)q_m} := y_{N}$. The knot sequence ${\ttt}$ is also extended with stacked knots, with $t_{-m+1} = \cdots = t_0$ and $t_{2m+(N-2)q_m} = \cdots = t_{3m+(N-2)q_m-1}$. With this setup, we define the knot sequences ${\ttt}_{-k}, \ k=0,\ldots,m-1,$ and ${\ttt}_{N+k}, \ k=0,\ldots,m-1,$ by
\begin{equation}
\left\{
\begin{aligned}
& {\ttt}_{-k} = \left\{ t_{-k}, \ldots, \underbrace{t_{0}, \ldots, t_{m}}_{\text{$m+1-k$ knots}} \right\}, \quad k=0,\ldots,m-1; \\
& {\ttt}_{N+k} = \left\{ \underbrace{t_{m+(N-2)q_m}, \ldots, t_{2m+(N-2)q_m}}_{\text{$m+1-k$ knots}}, \ldots, t_{2m+(N-2)q_m+k} \right\}, \ {k=0,\ldots,m-1,}
\end{aligned}
\right.
\label{eq_ttt_knots_half_m_even}
\end{equation}
where the $m+1-k$ knots in $[t_0,t_m]$ are chosen to be evenly spread out among $t_0,\ldots,t_m$; and the $m+1-k$ knots in $[t_{m+(N-2)q_m},t_{2m+(N-2)q_m}]$ are chosen to be evenly spread out in among $t_{m+(N-2)q_m},\ldots,t_{2m+(N-2)q_m}$. Lastly, the knot sequences ${\ttt}_1$ and ${\ttt}_{N-1}$ are defined by
\begin{equation*}
{\ttt}_1 = \left\{ t_{m-q_m}, \ldots, t_{m+q_m} \right\}; \quad {\ttt}_{N-1} = \left\{ t_{m+(N-3)q_m}, \ldots, t_{m+(N-1)q_m} \right\}.
\end{equation*}

We note that, with the above definitions, each ${\ttt}_k, \ k=-m+1,\ldots,1; \ N-1,\ldots,N+m-1,$ contain exactly $m+1$ knots.\\

Second, let $m$ be odd. With $r_m := (m+1)/{2}$, we insert $r_m-1$ equally spaced knots in $(y_{j}, y_{j+1})$ if $j$ is even, and $r_m-2$ equally spaced knots in $(y_{j},y_{j+1})$ if $j$ is odd, so that $t_{m+(j-1)r_m - \lfloor j/2 \rfloor} = y_j$ for $j=1,\ldots,N-1.$ Furthermore, we also insert $m-1$ evenly spaced knots $t_1,\ldots,t_{m-1}$ in the interval $\left(y_0,y_1\right)$, with $t_0 := y_0$, as well as $m-1$ evenly spaced knots $t_{m+(N-2)r_m-\lfloor (N-1)/2 \rfloor+1},\ldots,t_{2m+(N-2)r_m-\lfloor (N-1)/2 \rfloor -1}$ in the interval $\left(y_N,y_{N+1}\right)$, with $t_{2m+(N-2)r_m-\lfloor (N-1)/2 \rfloor} := y_{N}$. The knot sequence ${\ttt}$ is also extended with stacked knots, with $t_{-m+1} = \cdots = t_{0}$ and $t_{2m+(N-2)r_m-\lfloor (N-1)/2 \rfloor} = \cdots = t_{3m+(N-2)r_m-\lfloor (N-1)/2 \rfloor -1}$. With this setup, we define the knot sequences ${\ttt}_{-k}, \ k=0,\ldots,m-1,$ and ${\ttt}_{N+k}, \ k=0,\ldots,m-1,$ by
\begin{equation}
\left\{
\begin{aligned}
& {\ttt}_{-k} = \left\{ t_{-k}, \ldots, \underbrace{t_{0}, \ldots, t_{m}}_{\text{$m+1-k$ knots}} \right\}, \quad k=0,\ldots,m-1; \\
& {\ttt}_{N+k} = \left\{ \underbrace{t_{m+(N-2)r_m - \lfloor (N-1)/2 \rfloor}, \ldots, t_{2m+(N-2)r_m-\lfloor (N-1)/2 \rfloor}}_{\text{$m+1-k$ knots}}, \ldots, t_{2m+(N-2)r_m-\lfloor (N-1)/2 \rfloor +k} \right\}, \\
& \pushright{ \quad k=0,\ldots,m-1,}
\end{aligned}
\right.
\label{eq_ttt_knots_half_m_odd}
\end{equation}
where the $m+1-k$ knots in $[t_0,t_m]$ are chosen to be evenly spread out among $t_0,\ldots,t_m$; and the $m+1-k$ knots in $[t_{m+(N-2)r_m - \lfloor (N-1)/2 \rfloor},t_{2m+(N-2)r_m -\lfloor (N-1)/2 \rfloor}]$ are chosen to be evenly spread out among $t_{m+(N-2)r_m - \lfloor (N-1)/2 \rfloor},\ldots,t_{2m+(N-2)r_m -\lfloor (N-1)/2 \rfloor}$. Lastly, the knot sequences ${\ttt}_1$ and ${\ttt}_{N-1}$ are defined by
\begin{equation*}
{\ttt}_1 = \left\{ t_{m-r_m}, \ldots, t_{m+r_m-1} \right\},
\end{equation*}
and
\begin{equation*}
\left\{
\begin{aligned}
& {\ttt}_{N-1} = \left\{ t_{m+(N-3)r_m-\lfloor(N-2)/2 \rfloor}, \ldots, t_{m+(N-1)r_m - \lfloor (N-1)/2 \rfloor-1} \right\}, \quad  {\textup{if }} N {\textup{ is even}} ;\\
& {\ttt}_{N-1} = \left\{ t_{m+(N-3)r_m-\lfloor(N-2)/2 \rfloor}, \ldots, t_{m+(N-1)r_m - \lfloor (N-1)/2 \rfloor} \right\}, \quad  {\textup{if }} N {\textup{ is odd}}.
\end{aligned}
\right.
\end{equation*}

Again, we note that with the above definitions, each ${\ttt}_k, \ k=-m+1,\ldots,1; \ N-1,\ldots,N+m-1,$ contain exactly $m+1$ knots.

\begin{defn}
\label{defn_local_interpolant_knots_half}
The local interpolation operator ${\RR}_m$ is defined by
\begin{equation}
({\RR}_{m}f)(x) := \sum_{\ell=1}^{m-1}f^{(\ell)}(a)L_{m,-\ell}(x) + \sum_{i=0}^{N} f(y_{i})L_{m,i}(x) + \sum_{r=1}^{m-1}f^{(r)}(b)L_{m,N+r}(x),
\label{eq_RR_knots_half}
\end{equation}
in terms of the spline molecules
\begin{equation}
\left\{ \begin{aligned}
& L_{m,-\ell}(x) := \sum_{k=0}^{m-1}b_{m,-\ell,k}N_{{\ttt}_{-m+1+k},m}(x), \quad \ell=0,\ldots,m-1; \\
& L_{m,1}(x) := \frac{N_{{\ttt}_1,m}(x)}{N_{{\ttt}_1,m}(y_1)}; \\
& L_{m,i}(x) := \frac{N_{{\ttt},m,m+(i-2)q_m}(x)}{N_{{\ttt},m,m+(i-2)q_m}(y_i)}, \quad i=2,\ldots,N-2, \quad {\textup{if }} m {\textup{ is even;}} \\
& L_{m,i}(x) := \frac{N_{{\ttt},m,m+(i-2)r_m-\lfloor (i-1)/2 \rfloor}(x)}{N_{{\ttt},m,m+(i-2)r_m -\lfloor (i-1)/2 \rfloor}(y_i)}, \quad i=2,\ldots,N-2, \quad {\textup{if }} m {\textup{ is odd;}} \\
& L_{m,N-1}(x) := \frac{N_{{\ttt}_{N-1,m}}(x)}{N_{{\ttt}_{N-1,m}}(y_{N-1})}; \\
& L_{m,N+r}(x) := \sum_{k=0}^{m-1} b_{m,N+r,k} N_{{\ttt}_{N+k},m}(x), \quad r=0,\ldots,m-1,
\end{aligned}
\right.
\label{eq_L_defn_knots_half}
\end{equation}
with the coefficients $b_{m,-\ell,k}, \ k,\ell=0,\ldots,m-1,$ and $b_{m,N+r,k}, \ k, r=0,\ldots,m-1,$ determined by the conditions
\begin{equation}
L_{m,-\ell}^{(n)}(a) = \delta_{\ell-n}; \quad 
L_{m,N+r}^{(n)}(b) = \delta_{r-n}, \quad \ell, r, \ n=0,\ldots,m-1.
\label{eq_b_cond_knots_half}
\end{equation}
\end{defn}

It is clear that the above molecules are compactly supported to ensure a local formulation, with
\begin{equation}
\left\{ \begin{aligned}
& {\supp}L_{m,-\ell} = [y_{0}, y_{1}], \quad \ell=0,\ldots,m-1;\\
& {\supp}L_{m,1} = [t_{m-q_m}, y_2], \ {\textup{if }} m \textup{ is even}; \quad {\supp}L_{m,1} = [t_{m-r_m}, y_2], \ {\textup{if }} m \textup{ is odd};\\
& {\supp}L_{m,i} = [y_{i-1}, y_{i+1}], \quad i=2,\ldots,N-2;\\
& {\supp}L_{m,N-1} = [y_{N-2}, t_{m+(N-1)q_m}], \quad {\textup{if }} m \textup{ is even};\\
& {\supp}L_{m,N-1} = [y_{N-2}, t_{m+(N-1)r_m - \lfloor (N-1)/2 \rfloor-1} ], \ {\textup{if }} m {\textup{ is odd}}, \ N {\textup{ is even}};\\
& {\supp}L_{m,N-1} = [y_{N-2}, t_{m+(N-1)r_m - \lfloor (N-1)/2 \rfloor} ], \ {\textup{if }} m {\textup{ is odd}}, \ N {\textup{ is odd}};\\
& {\supp}L_{m,N+r} = [y_{N-1}, y_{N}], \quad r=0,\ldots,m-1.
\end{aligned}
\right.
\label{eq_L_support_knots_half}
\end{equation}

From the construction in (\ref{eq_L_defn_knots_half}), it is clear that
\begin{equation}
L_{m,i}(y_j) = \delta_{i-j}, \quad i=1,\ldots,N-1; \ j=0,\ldots,N.
\label{eq_L_interpolates_knots_half}
\end{equation}

By using also (\ref{eq_b_cond_knots_half}), the following result follows immediately.

\begin{theorem}
\label{thm_local_interpolant_knots_half}
The local interpolation operator ${\RR}_m$, formulated in (\ref{eq_RR_knots_half}) in Definition \ref{defn_local_interpolant_knots_half}, satisfies the Hermite interpolation conditions
\begin{gather*}
({\RR}_m f)(y_i) = f(y_i), \quad i=0,\ldots,N;\\
({\RR}_m f)^{(n)}(a) = f^{(n)}(a), \quad ({\RR}_m f)^{(n)}(b) = f^{(n)}(b), \quad n=1,\ldots,m-1.
\end{gather*}
\end{theorem}


\section{Blending interpolation}
\label{sec_blending}

With ${\QQ}_m$ and ${\RR}_m$ formulated in Definitions \ref{defn_quasi-interpolant_knots_half} and \ref{defn_local_interpolant_knots_half}, respectively, we can now derive the blending operator as in (\ref{eq_PP})-(\ref{eq_RR+QQ}). More specifically, we have
\begin{align}
{\PP}_{m} &= {\QQ}_{m} + {\RR}_{m} - {\RR}_{m}{\QQ}_{m} \nonumber \\
& = \sum_{\ell=1}^{m-1}f^{(\ell)}(a) M_{m,-\ell}(x) + \sum_{\ell=1}^{m-1} \left[ f^{(\ell)}(a) - ({\QQ}_m f)^{(\ell)}(a) \right] L_{m,-\ell}(x) \nonumber \\
& \quad \quad + \sum_{i=0}^{N}f(y_i) M_{m,i}(x) + \sum_{i=0}^{N} \left[ f(y_i) - ({\QQ}_mf)(y_i) \right] L_{m,i}(x) \nonumber \\
& \quad \quad + \sum_{r=1}^{m-1}f^{(r)}(b) M_{m,N+r}(x) + \sum_{r=1}^{m-1} \left[ f^{(r)}(b) - ({\QQ}_mf)^{(r)}(b) \right] L_{m,N+r}(x).
\label{eq_PP_knots_half}
\end{align}

We can then show that ${\PP}_m$ satisfies all four conditions (\ref{it_blending_cond_i})-(\ref{it_blending_cond_iv}) of Section \ref{sec_intro}, as follows:

\begin{theorem}
\label{thm_PP_conditions}
The blending operator ${\PP}_m$, defined by (\ref{eq_PP_knots_half}), is local and satisfies the polynomial preservation property of the quasi-interpolation operator ${\QQ}_m$ as well as the Hermite interpolation conditions of the local interpolation operator ${\RR}_m$; that is, ${\PP}_m$ satisfies (\ref{eq_poly_preserve})-(\ref{eq_hermite_interpolatory_cond}).
\end{theorem}

\begin{proof}
First, from the support properties (\ref{eq_M_support_knots_half}) and (\ref{eq_L_support_knots_half}), it is clear that ${\PP}_m$ is local. Second, the polynomial preservation property (\ref{eq_poly_preserve}) follows easily from Theorem \ref{thm_quasi_knots_half}. The interpolation property (\ref{eq_interpolatory_cond}) follows directly from Theorem \ref{thm_local_interpolant_knots_half}, together with (\ref{eq_RR_knots_half}), (\ref{eq_L_interpolates_knots_half}) and (\ref{eq_b_cond_knots_half}). Lastly, we observe that
\begin{gather*}
({\RR}_{m}({\QQ}_{m}f))^{(n)}(a) = ({\QQ}_{m}f)^{(n)}(a); \quad n=1,\ldots,m-1; \\
({\RR}_{m}({\QQ}_{m}f))^{(n)}(b) = ({\QQ}_{m}f)^{(n)}(b), \quad n=1,\ldots,m-1,
\end{gather*}
from the construction of ${\RR}_{m}$ and the spline molecules $L_{m,i}$ in Definition \ref{defn_local_interpolant_knots_half}. The result then follows using Theorem \ref{thm_local_interpolant_knots_half}.
\end{proof}

\section{Approximation order}
\label{sec_blending_approx_order}

Lastly, we provide an error analysis of the blending interpolation operators ${\PP}_m$. From the definition of ${\PP}_m$ in (\ref{eq_PP_knots_half}) it is clear that, in order to bound the error of spline interpolation of these operators, we need upper bounds on the spline molecules $M_{m,i}, \ L_{m,i}$ and $M_{m,i}^{(n)}$.

\subsection{Upper bounds on $M_{m,i}$}
Our first task is to find upper bounds on the spline coefficients $a_{m,i,j}$ in Definition \ref{defn_quasi-interpolant_knots_half}. To this end, we start by noting that, for an integer $m\geq3$ and a sequence of real numbers $\left\{x_1,\ldots,x_{m-1}\right\}$,
\begin{equation}
\left\{
\begin{aligned}
& \sum_{i=1}^{m-1} \sigma^{\ell}(x_1,\ldots,x_{i-1},x_{i+1},\ldots,x_{m-1}) = (m-1-\ell)\sigma^{\ell}(x_1,\ldots,x_{m-1});\\
& \sum_{i=1}^{m-1} x_i\sigma^{\ell}(x_1,\ldots,x_{i-1},x_{i+1},\ldots,x_{m-1}) = (\ell+1)\sigma^{\ell+1}(x_1,\ldots,x_{m-1}),
\end{aligned}
\right.
\label{eq_symm_poly_1_2}
\end{equation}
for $\ell=0,\ldots,m-2,$ and
\begin{equation}
x_{m-1}\sigma^{m-2-\ell}(x_1,\ldots,x_{m-2}) + \sigma^{m-1-\ell}(x_1,\ldots,x_{m-2}) = \sigma^{m-1-\ell}(x_1,\ldots,x_{m-1}),
\label{eq_symm_poly_3}
\end{equation}
for $\ell=0,\ldots,m-1$, all of which follow directly from the definition of the symmetric polynomials in (\ref{eq_sigma}). We will rely on the following lemma, which originally appeared in \cite[Lemma~2.1]{chuistudents}. It is given here with a modified proof.

\begin{lemma}
For an integer $m\geq 3$, let $\left\{ x_1,\ldots, x_{m-1}\right\}$ and $\left\{y_1,\ldots,y_{m-1}\right\}$ denote two sequences of real numbers. Then
\begin{multline}
\sum_{\ell=0}^{m-1} (-1)^{\ell} \frac{\sigma^{m-1-\ell}(x_1,\ldots,x_{m-1})}{\binom{m-1}{\ell}} \sigma^{\ell}(y_1,\ldots,y_{m-1}) \\ = \frac{1}{(m-1)!} \left[ \sum_{1\leq t_1,\ldots,t_{m-1}\leq m-1}\prod_{k=1}^{m-1} \left(x_{t_k}-y_k\right)\right].
\label{eq_symm_poly_4}
\end{multline}
\label{lemma_symm_poly}
\end{lemma}

\begin{proof}
Our proof is by induction on $m$. It can be verified directly, using the definition (\ref{eq_sigma}) of the symmetric polynomials, that the result holds for $m=3$. We now assume the result holds for an integer $m-1$ and proceed to prove (\ref{eq_symm_poly_4}). Using the induction hypothesis, (\ref{eq_symm_poly_1_2}) and (\ref{eq_symm_poly_3}), we have that
\begin{align*}
& \frac{1}{(m-1)!} \left[ \sum_{1\leq t_1,\ldots,t_{m-1}\leq m-1}\prod_{k=1}^{m-1} \left(x_{t_k}-y_k\right)\right] \\
& = \frac{1}{(m-1)!} \left[ \sum_{\substack{1\leq t_2,\ldots,t_{m-1} \leq m-2 \\ t_1=m-1}}\prod_{k=1}^{m-1} \left(x_{t_k}-y_k\right) + \sum_{\substack{1\leq t_1,t_3\ldots,t_{m-1} \leq m-2 \\ t_2=m-1}}\prod_{k=1}^{m-1} \left(x_{t_k}-y_k\right)\right. \\
& \pushright{\left. + \cdots + \sum_{\substack{1\leq t_1,\ldots,t_{m-2} \leq m-2 \\ t_{m-1}=m-1}}\prod_{k=1}^{m-1} \left(x_{t_k}-y_k\right) \right]} \\
& = \frac{1}{(m-1)!} \left[ \left(x_{m-1}-y_1\right)\sum_{1\leq t_2,\ldots,t_{m-1} \leq m-2}\prod_{\substack{k=2}}^{m-1} \left(x_{t_k}-y_k\right) \right. \\
& \quad \quad \quad \quad \quad \quad \quad \quad { \left. + \left(x_{m-1}-y_2\right)\sum_{1\leq t_1,t_3\ldots,t_{m-1} \leq m-2}\prod_{\substack{k=1 \\ k\neq 2}}^{m-1} \left(x_{t_k}-y_k\right)\right.} \\
& \pushright{\left. + \cdots + \left(x_{m-1}-y_{m-1}\right)\sum_{1\leq t_1,\ldots,t_{m-2} \leq m-2}\prod_{\substack{k=1}}^{m-2} \left(x_{t_k}-y_k\right) \right]} \\
& = \frac{1}{m-1} \left[ \left(x_{m-1}-y_1\right) \sum_{\ell=0}^{m-2}(-1)^{\ell}\frac{\sigma^{m-2-\ell}(x_1,\ldots,x_{m-2})}{\binom{m-2}{\ell}} \sigma^{\ell}(y_2,\ldots,y_{m-1}) \right. \\
&  \quad \quad \quad \quad { \left. + \left(x_{m-1}-y_2\right)\sum_{\ell=0}^{m-2}(-1)^{\ell}\frac{\sigma^{m-2-\ell}(x_1,\ldots,x_{m-2})}{\binom{m-2}{\ell}} \sigma^{\ell}(y_1,y_3,\ldots,y_{m-1}) \right.} \\
& \pushright{\left. + \cdots + \left(x_{m-1}-y_{m-1}\right)\sum_{\ell=0}^{m-2}(-1)^{\ell}\frac{\sigma^{m-2-\ell}(x_1,\ldots,x_{m-2})}{\binom{m-2}{\ell}} \sigma^{\ell}(y_1,\ldots,y_{m-2}) \right]} \\
& = \frac{1}{m-1} \left[ \sum_{\ell=0}^{m-2}(-1)^{\ell}\frac{x_{m-1}\sigma^{m-2-\ell}(x_1,\ldots,x_{m-2})}{\binom{m-2}{\ell}} \times \right. \\
& \pushright{\left. \left( \sigma^{\ell}(y_2,\ldots,y_{m-1}) + \sigma^{\ell}(y_1,y_3,\ldots,y_{m-1}) + \cdots + \sigma^{\ell}(y_1,\ldots,y_{m-2}) \right) \vphantom{\sum_{\ell=0}^{m-2}} \right.} \\
& \quad \quad \quad \quad \left. - \sum_{\ell=0}^{m-2}(-1)^{\ell}\frac{\sigma^{m-2-\ell}(x_1,\ldots,x_{m-2})}{\binom{m-2}{\ell}} \vphantom{\sum_{\ell=0}^{m-2}} \times \right. \\
& \pushright{ \left. \left( y_1\sigma^{\ell}(y_2,\ldots,y_{m-1}) + y_2\sigma^{\ell}(y_1,y_3,\ldots,y_{m-1}) + \cdots + y_{m-1}\sigma^{\ell}(y_1,\ldots,y_{m-2}) \right) \vphantom{\sum_{\ell=0}^{m-2}}
\right]}\\
& = \frac{1}{m-1} \left[ \sum_{\ell=0}^{m-1}(-1)^{\ell}\frac{x_{m-1}\sigma^{m-2-\ell}(x_1,\ldots,x_{m-2})}{\binom{m-2}{\ell}} (m-1-\ell) \sigma^{\ell}(y_1,\ldots,y_{m-1}) \right. \\
& \pushright{ \left. + \sum_{\ell=-1}^{m-2}(-1)^{\ell+1}\frac{\sigma^{m-2-\ell}(x_1,\ldots,x_{m-2})}{\binom{m-2}{\ell}} (\ell+1) \sigma^{\ell+1}(y_1,\ldots,y_{m-1}) \right] }\\
& = \sum_{\ell=0}^{m-1}(-1)^{\ell}\frac{x_{m-1}\sigma^{m-2-\ell}(x_1,\ldots,x_{m-2})}{\binom{m-1}{\ell}} \sigma^{\ell}(y_1,\ldots,y_{m-1}) \\
& \pushright{ + \sum_{\ell=0}^{m-1}(-1)^{\ell}\frac{\sigma^{m-1-\ell}(x_1,\ldots,x_{m-2})}{\binom{m-1}{\ell}} \sigma^{\ell}(y_1,\ldots,y_{m-1}) }\\
&= \sum_{\ell=0}^{m-1} (-1)^{\ell} \frac{\sigma^{m-1-\ell}(x_1,\ldots,x_{m-1})}{\binom{m-1}{\ell}} \sigma^{\ell}(y_1,\ldots,y_{m-1}),
\end{align*}
completing our inductive proof of (\ref{eq_symm_poly_4}).
\end{proof}

Next, we recall the standard formula for the expansion of a linear factorization of a monomial in terms of the symmetric polynomials: for a sequence of real numbers $\left\{ t_1,\ldots, t_n \right\}$ and some $r\in\R$ and $n\in\N$,
\begin{equation}
\prod_{j=1}^{n} \left(r - t_j \right) = \sum_{j=0}^{n} (-1)^{j} r^{n-j}\sigma^{j}(t_1,\ldots,t_n).
\label{eq_symm_poly_5}
\end{equation}

The following lemma originally appeared in \cite[Theorem~2.2]{chuistudents}, where the result was proved only for an unbounded interval. Our lemma below is a non-trivial extension that include upper bounds on the spline coefficients near the boundaries $x=a$ and $x=b$.

\begin{lemma}
For an integer $m\geq3$, let ${\xxx}$ and ${\yyy}$ be the sequences defined in (\ref{eq_xxx}) and (\ref{eq_yyy_knots_half}), respectively. Suppose that
\begin{equation}
\gamma := \sup_{n=0,\ldots,N+m-1} |x_n - y_{n-m+1}|; \quad
\delta := \min \left\{1, \inf_{n=0,\ldots,N-1} |y_{n+1} - y_n| \right\},
\label{eq_gamma_delta_H}
\end{equation}
with the definition that $y_{-m+1} = \cdots = y_0$. Then
\begin{equation}
|a_{m,i,j}| \leq \frac{1}{(m-2)!}\left(\frac{\gamma}{\delta}\right)^{m-1},
\label{eq_a_H_upper_bound}
\end{equation}
for $i=-m+1,\ldots,N+m-1, \ j=0,\ldots,m-1$.
\label{lemma_upper_bound_a}
\end{lemma}

\begin{proof}
First, if $i\in\left\{m-1,\ldots,N+1-m\right\}$, $j\in\left\{0,\ldots,m-1\right\}$ or $i\in\left\{0,\ldots,m-2\right\}$, $j\in\left\{m-1-i,\ldots,m-1\right\},$ or $i\in\left\{N-m+2,\ldots,N\right\}$, $j\in\left\{0,\ldots,N-i\right\}$, the result follows as in \cite[Theorem~2.2]{chuistudents}. We proceed to prove (\ref{eq_a_H_upper_bound}) for the spline coefficients near the left hand side endpoint; the proof for the spline coefficients near the right hand side boundary is similar. To this end, let $i\in\left\{1,\ldots,m-2\right\}$ and $j\in\left\{0,\ldots,m-2-i\right\}$ be fixed. An explicit formulation of the confluent Vandermonde determinant $D_{C}(y_0,y_0^{(1)},\ldots,y_0^{(m-1-j-i)}, y_1,\ldots,y_{i+j})$ is given in \cite{horn37topics,papanicolaou2014some}, namely
\begin{equation}
D_{C}(y_0,y_0^{(1)},\ldots,y_0^{(m-1-j-i)}, y_1,\ldots,y_{i+j}) = \prod_{\ell=1}^{i+j}(y_{\ell}-y_0)^{m-j-i} \prod_{1\leq k<\ell \leq i+j} (y_{\ell}-y_k).
\label{eq_a_proof_3}
\end{equation}
With the definition ${\mathbf{r}}=[1,r,r^2,\ldots,r^{m-1}]^{T}$ for some $r\in\R$ and with $y_{-m+1} = \cdots = y_0$, we now have (following the notation introduced in Section \ref{sec_quasi})
\begin{align*}
& D_{C}(y_0,y_0^{(1)},\ldots,y_0^{(m-1-j-i)}, y_1,\ldots,y_{i-1},{\mathbf{r}}, y_{i+1},\ldots,y_{i+j}) \\
&= \prod_{\substack{\ell=1\\\ell\neq i}}^{i+j}(y_{\ell}-y_0)^{m-j-i} \prod_{\substack{1\leq k<\ell \leq i+j\\k,\ell\neq i}} (y_{\ell}-y_k){ \left(r-y_0\right)^{m-j-i} \prod_{k=1}^{i-1}\left(r-y_k\right) \prod_{k=i+1}^{i+j} \left(y_k - r \right) } \nonumber \\
& = \prod_{\substack{\ell=1\\\ell\neq i}}^{i+j}(y_{\ell}-y_0)^{m-j-i} \prod_{\substack{1\leq k<\ell \leq i+j\\k,\ell\neq i}} (y_{\ell}-y_k) \prod_{\substack{k=i+j-m+1\\k\neq i}}^{i+j}(-1)^{j}\left(r-y_k\right) \\
& = (-1)^{j}\prod_{\substack{\ell=1\\\ell\neq i}}^{i+j}(y_{\ell}-y_0)^{m-j-i} \prod_{\substack{1\leq k<\ell \leq i+j\\k,\ell\neq i}} (y_{\ell}-y_k) { \prod_{\substack{k=0\\k\neq m-1-j}}^{m-1}\left(r-y_{k+i+j-m+1}\right) }\\
& = (-1)^{j}\prod_{\substack{\ell=1\\\ell\neq i}}^{i+j}(y_{\ell}-y_0)^{m-j-i} \prod_{\substack{1\leq k<\ell \leq i+j\\k,\ell\neq i}} (y_{\ell}-y_k) \times  \\
& \pushright{  \sum_{n=0}^{m-1} (-1)^{n} r^{m-1-n}\sigma^{n}(y_{i+j-m+1},\ldots,y_{i-1},y_{i+1},\ldots,y_{i+j}) }
\end{align*}
from (\ref{eq_symm_poly_5}), so that, from the definition (\ref{eq_xi1})-(\ref{eq_xi2}) and (\ref{eq_symm_poly_4}) in Lemma \ref{lemma_symm_poly},
\begin{align}
& D_{C}(y_0,y_0^{(1)},\ldots,y_0^{(m-1-j-i)}, y_1,\ldots,y_{i-1},\xi_{{\xxx},m,i+j-m+1}, y_{i+1},\ldots,y_{i+j}) \nonumber \\
& = (-1)^{j}\prod_{\substack{\ell=1\\\ell\neq i}}^{i+j}(y_{\ell}-y_0)^{m-j-i} \prod_{\substack{1\leq k<\ell \leq i+j\\k,\ell\neq i}} (y_{\ell}-y_k) \times \nonumber \\
& \pushright{  \sum_{n=0}^{m-1} (-1)^{n} \xi^{m-1-n}_{\xxx}(i+j-m+1)\sigma^{n}(y_{i+j-m+1},\ldots,y_{i-1},y_{i+1},\ldots,y_{i+j}) } \nonumber \\
& = (-1)^{j}\prod_{\substack{\ell=1\\\ell\neq i}}^{i+j}(y_{\ell}-y_0)^{m-j-i} \prod_{\substack{1\leq k<\ell \leq i+j\\k,\ell\neq i}} (y_{\ell}-y_k) \times \nonumber \\
& \pushright{ \sum_{n=0}^{m-1}(-1)^{n}\frac{\sigma^{m-1-n}(x_{i+j-m+2},\ldots,x_{i+j})}{\binom{m-1}{n}}\sigma^{n}(y_{i+j-m+1},\ldots,y_{i-1},y_{i+1},\ldots,y_{i+j}) } \nonumber\\
& = (-1)^{j}\prod_{\substack{\ell=1\\\ell\neq i}}^{i+j}(y_{\ell}-y_0)^{m-j-i} \prod_{\substack{1\leq k<\ell \leq i+j\\k,\ell\neq i}} (y_{\ell}-y_k) \times \nonumber \\
& \pushright{  \frac{1}{(m-1)!} \sum_{i+j-m+2 \leq t_{i+j-m+1},\ldots,t_{i-1},t_{i+1},\ldots,t_{i+j}\leq i+j} \prod_{\substack{n=i+j-m+1\\n\neq i}}^{i+j} \left(x_{t_n}-y_n\right) }.\nonumber\\
\label{eq_a_proof_4}
\end{align}
It therefore follows from the definition of $a_{m,i,j}$ in (\ref{eq_a_lb1_half}), as well as (\ref{eq_a_proof_4}) and the definitions in (\ref{eq_gamma_delta_H}), that
\begin{align*}
|a_{m,i,j}| &\leq \frac{1}{(m-1)!}|y_{i}-y_0|^{-(m-j-i)}\prod_{\substack{k=1\\k\neq i}}^{i+j} |y_{i}-y_k|^{-1} \times \\
& {\quad \quad  \sum_{i+j-m+2 \leq t_{i+j-m+1},\ldots,t_{i-1},t_{i+1},\ldots,t_{i+j}\leq i+j} \prod_{\substack{n=i+j-m+1\\n\neq i}}^{i+j} |x_{t_n}-y_n|}  \leq \frac{1}{(m-2)!}\left(\frac{\gamma}{\delta}\right)^{m-1}.
\end{align*}

Lastly, let $i=0, \ j\in\left\{0,\ldots,m-2\right\}$ or $i=0, \ \ell\in\left\{1,\ldots,m-1\right\}, \ j\in\left\{0,\ldots,m-1-\ell \right\}$ be fixed. Then, with ${\mathbf{r}}=[1,r,r^2,\ldots,r^{m-1}]^{T}$ for some $r\in\R$ and $y_{-m+1} = \cdots = y_0$, we have, from (\ref{eq_a_proof_3}) with $i=0$,
\begin{align*}
& D_{C}(y_0,y_0^{(1)},\ldots,y_0^{(\ell-1)}, {\mathbf{r}}, y_0^{(\ell+1)}, \ldots, y_0^{(m-1-j)}, y_1,\ldots,y_{j})  \\
&= \prod_{s=1}^{j}(y_{s}-y_0)^{m-j-1} \prod_{1\leq k<s \leq j} (y_{s}-y_k) { \left(r-y_0\right)^{m-j-1} \prod_{k=1}^{j} \left(y_k - r \right) } \\
& = \prod_{s=1}^{j}(y_{s}-y_0)^{m-j-1} \prod_{1\leq k<s \leq j} (y_{s}-y_k) \prod_{k=j-m+2}^{j}(-1)^{j}\left(r-y_k\right)  \\
& = (-1)^{j}\prod_{s=1}^{j}(y_{s}-y_0)^{m-j-1} \prod_{1\leq k<s \leq j} (y_{s}-y_k) { \sum_{n=0}^{m-1} (-1)^{n} r^{m-1-n}\sigma^{n}(y_{j-m+2},\ldots,y_{j}), }
\end{align*}
so that, from (\ref{eq_xi1})-(\ref{eq_xi2}), and (\ref{eq_symm_poly_5}) and (\ref{eq_symm_poly_4}) in Lemma \ref{lemma_symm_poly},
\begin{align}
& D_{C}(y_0,y_0^{(1)},\ldots,y_0^{(\ell-1)}, \xi_{{\xxx},m,j-m+1}, y_0^{(\ell+1)}, \ldots, y_0^{(m-1-j)}, y_1,\ldots,y_{j}) \nonumber \\
& = (-1)^{j}\prod_{s=1}^{j}(y_{s}-y_0)^{m-j-1} \prod_{1\leq k<s \leq j} (y_{s}-y_k) { \sum_{n=0}^{m-1} (-1)^{n} \xi^{m-1-n}_{\xxx}(j-m+1)\sigma^{n}(y_{j-m+2},\ldots,y_{j}) } \nonumber \\
& = (-1)^{j}\prod_{s=1}^{j}(y_{s}-y_0)^{m-j-1} \prod_{1\leq k<s \leq j} (y_{s}-y_k) \times \nonumber \\
& \pushright{ \sum_{n=0}^{m-1}(-1)^{n}\frac{\sigma^{m-1-n}(x_{j-m+2},\ldots,x_{j})}{\binom{m-1}{n}}\sigma^{n}(y_{j-m+2},\ldots,y_{j}) } \nonumber\\
& = (-1)^{j}\prod_{s=1}^{j}(y_{s}-y_0)^{m-j-1} \prod_{1\leq k<s \leq j} (y_{s}-y_k) {  \frac{1}{(m-1)!} \sum_{j-m+2 \leq t_{j-m+2},\ldots,t_{j}\leq j} \prod_{n=j-m+2}^{j} \left(x_{t_n}-y_n\right) }.\nonumber\\
\label{eq_a_proof_5}
\end{align}
It therefore follows from the definition of $a_{m,i,j}$ in (\ref{eq_a_lb1_half})-(\ref{eq_a_lb2_half}), as well as (\ref{eq_a_proof_5}) and the definitions in (\ref{eq_gamma_delta_H}), that
\begin{align*}
|a_{m,i,j}| &\leq \frac{1}{(m-1)!}\prod_{s=1}^{j} |y_{s}-y_0|^{-1} \sum_{j-m+2 \leq t_{j-m+2},\ldots,t_{j}\leq j} \prod_{n=j-m+2}^{j} |x_{t_n}-y_n| \\
& \leq \frac{1}{(m-2)!}\left(\frac{\gamma}{\delta}\right)^{m-1}. \qedhere
\end{align*}
\end{proof}

From Lemma \ref{lemma_upper_bound_a}, and the fact that the B-splines $\left\{ N_{{\xxx},m,j}: \ j=-m+1,\ldots,N \right\}$ provide a partition of unity, the following upper bound on the spline molecules $M_{m,i}$ follows easily:

\begin{theorem}
For an integer $m\geq3$,
\begin{equation}
|M_{m,i}(x)| \leq \frac{1}{(m-2)!} \left( \frac{\gamma}{\delta} \right)^{m-1},
\label{eq_M_H_upper_bound}
\end{equation}
for all $i=-m+1,\ldots,N+m-1,$ and $x\in[a,b]$, where the constants $\gamma$ and $\delta$ are defined in (\ref{eq_gamma_delta_H}).
\label{thm_upper_bound_M}
\end{theorem}

\begin{proof}
Using the spline molecule definition (\ref{eq_M_defn_knots_half}) and (\ref{eq_a_H_upper_bound}) in Lemma \ref{lemma_upper_bound_a}, and the properties of the normalized B-splines (see, for example, \cite[Theorem~6.4]{chui1997wavelets}), we have, for $i=-m+1,\ldots,N+m-1,$
\begin{equation*}
|M_{m,i}(x)| \leq \sum_{j=0}^{m-1} |a_{m,i,j} N_{{\xxx},m,i+j-m+1}(x)| \leq \frac{1}{(m-2)!}\left(\frac{\gamma}{\delta}\right)^{m-1}. \qedhere
\end{equation*}
\end{proof}

\subsection{Upper bounds on $M_{m,i}^{(n)}$}
Lemma \ref{lemma_upper_bound_a} also leads to the following upper bound on the spline molecule derivatives $M_{m,i}^{(n)}$:

\begin{theorem}
For an integer $m\geq3$, let $M_{m,i}$, $i=-m+1,\ldots,N+m-1,$ be defined by (\ref{eq_M_defn_knots_half}), respectively. Then
\begin{equation}
\left\{
\begin{aligned}
& |M_{m,i}^{(n)}(a)| \leq \frac{m}{(m-2)!}\left(\frac{\gamma}{\delta}\right)^{m-1}\left(\frac{2}{\delta}\right)^{m-1},\\
& \quad \quad \quad \quad \quad \quad \quad \quad {i=-m+1,\ldots,m-1; \ n=1,\ldots,m-1;}\\
& |M_{m,i}^{(n)}(b)| \leq \frac{m}{(m-2)!}\left(\frac{\gamma}{\delta}\right)^{m-1}\left(\frac{2}{\delta}\right)^{m-1},\\
& \pushright{i=N-m+2,\ldots,N+m-1; \ n=1,\ldots,m-1.}\\
\end{aligned}
\right.
\label{eq_M_H_der_upper_bound}
\end{equation}
\label{thm_upper_bound_M_der}
\end{theorem}

\begin{proof}
Let $i\in\left\{ -m+1,\ldots,m-1\right\}$ be fixed (the proof for the right hand side boundary follows similarly). From the spline molecule definition (\ref{eq_M_defn_knots_half}), the upper bound (\ref{eq_a_H_upper_bound}), the recursive formulation of the derivative of a B-spline and the definition of $\delta$ in (\ref{eq_gamma_delta_H}), we have, for any $n\in\left\{1,\ldots,m-1\right\}$,
\begin{align*}
|M_{m,i}^{(n)}(a)| & \leq \sum_{j=0}^{m-1} |a_{m,i,j} N_{{\xxx},m,i+j-m+1}^{(n)}(a)| \\
& \leq \frac{1}{\delta^n} \frac{1}{(m-2)!} \left(\frac{\gamma}{\delta}\right)^{m-1} \sum_{j=0}^{m-1} \sum_{k=0}^{n} \binom{n}{k} N_{{\xxx},m-n,i+j-m+1+k}(a) \\
& \leq \frac{1}{\delta^n} \frac{1}{(m-2)!} \left(\frac{\gamma}{\delta}\right)^{m-1} m 2^{m-1} \leq \frac{m}{(m-2)!}\left(\frac{\gamma}{\delta}\right)^{m-1}\left(\frac{2}{\delta}\right)^{m-1}. \qedhere
\end{align*}
\end{proof}

\subsection{Upper bounds on $L_{m,i}$}
We now proceed to derive upper bounds on $\ L_{m,i}$. We will rely on the following lemma.

\begin{lemma}
For an integer $m\geq3$, define
\begin{equation}
\rho := \max \left\{ 1, |y_1-y_0|, |y_N-y_{N-1}| \right\}; \quad
\lambda := \min \left\{ 1, \inf |t_{n+1}-t_n| \right\}.
\label{eq_rho_lambda}
\end{equation}
Then, for $i,j=0,\ldots,m-1,$
\begin{equation}
\frac{m-1}{\rho^{m-1}} \leq |N_{{\ttt}_{-m+1+j},m}^{(i)}(a)|, \ |N_{{\ttt}_{N+j},m}^{(i)}(b)| \leq \frac{(m-1)!(m-1)!}{\lambda^{m-1}}.
\label{eq_bound_b_spline_der_scheme_H}
\end{equation}
\label{lemma_bind_boundary_b_splines}
\end{lemma}

\begin{proof}
We provide the proof for when $x=a$; the proof for the case where $x=b$ follows similarly. By using the recursive formula for the derivative of a B-spline and the positivity property of the normalized B-splines (see, for example, \cite[Theorem~6.4]{chui1997wavelets}), together with the definitions of the knot sequences ${\ttt}_{-m+1},\ldots,{\ttt}_0$ in (\ref{eq_ttt_knots_half_m_even}) and (\ref{eq_ttt_knots_half_m_odd}), that, for $i,j\in\left\{ 1,\ldots,m-1\right\}$,
\begin{align*}
& N_{{\ttt}_{-m+1+j},m}^{(i)}(a) \\
& = (m-1)\left[ \frac{1}{t_{j}-y_0} N_{{\ttt}_{-m+1+j},m-1,-m+1+j}^{(i-1)}(a) \right. \\
& \pushright{ \left. - \frac{1}{y_1-t_{-m+2+j}} N_{{\ttt}_{-m+1+j},m-1,-m+2+j}^{(i-1)}(a) \right] } \\
& = (m-1)(m-2) \left[  \frac{1}{(t_{j}-y_0)(t_{j-1}-y_0)}N_{{\ttt}_{-m+1+j},m-2,-m+1+j}^{(i-2)}(a) \right. \\
& \quad \quad\quad \quad{ \left. - 2\frac{1}{(t_{j}-y_0)(y_1-t_{-m+2+j})}N_{{\ttt}_{-m+1+j},m-2,-m+2+j}^{(i-2)}(a) \right.} \\
& \pushright{ \left. + \frac{1}{(y_1-t_{-m+2+j})(y_1-t_{-m+3+j})} N_{{\ttt}_{-m+1+j},m-2,-m+3+j}^{(i-2)}(a) \right] }\\
& = \cdots \\
& = (m-1)(m-2)\cdots(m-i) \times \\
& \quad \quad \quad \quad \left[ \frac{1}{(t_{j}-y_0)\cdots(t_{j-i+1}-y_0)}N_{{\ttt}_{-m+1+j},m-i,-m+1+j}(a) - \cdots \right. \\
& \quad \quad \quad \quad \quad { \left. + \frac{1}{(y_1-t_{-m+1+j+1})\cdots(y_1-t_{-m+1+j+i})}N_{{\ttt}_{-m+1+j},m-i,-m+1+j+i}(a)  \right]. }
\end{align*}
The only B-spline in this sum with a non-zero value at $x=a$ is $N_{{\ttt}_{-m+1+j},m-i,-m+i+1}(a) = 1$. Therefore, it follows that
\begin{equation*}
|N_{{\ttt}_{-m+1+j},m}^{(i)}(a)| \geq \frac{(m-1)(m-2)\cdots(m-i)}{\rho^{i}}\binom{i}{j} \geq \frac{m-1}{\rho^{m-1}},
\end{equation*}
while
\begin{equation*}
|N_{{\ttt}_{-m+1+j},m}^{(i)}(a)| \leq \frac{(m-1)(m-2)\cdots(m-i)}{\lambda^{i}}\binom{i}{j} \leq \frac{(m-1)!(m-1)!}{\lambda^{m-1}}. \qedhere
\end{equation*}
\end{proof}

We are now in a position to derive upper bounds on the spline coefficients $b_{m,i,k}$ in (\ref{eq_L_defn_knots_half}).

\begin{lemma}
For an integer $m\geq3$, let $\rho$ and $\lambda$ be defined by (\ref{eq_rho_lambda}), and set
\begin{equation}
\tau := \max\left\{ 1, \frac{\rho^{m-1}}{m-1}, m!\left(\frac{\rho^{m-1}(m-1)!(m-2)!}{\lambda^{m-1}}\right)^m \right\}.
\label{eq_tau}
\end{equation}
Then
\begin{equation}
|b_{m,i,k}| \leq \tau
\label{eq_b_H_upper_bound}
\end{equation}
for all $i=-m+1,\ldots,-1,N+1,\ldots,N+m-1$ and $k=0,\ldots,m-1$.
\label{lemma_upper_bound_b}
\end{lemma}

\begin{proof}
Let $-\ell:=i\in\left\{-m+1,\ldots,0\right\}$ and $k\in\left\{0,\ldots,m-1\right\}$ be fixed (the proof for the case when $i\in\left\{N,\ldots,N+m-1\right\}$, $k\in\left\{0,\ldots,m-1\right\}$ is similar). From the linear system in the first line of (\ref{eq_L_defn_knots_half}) and the first equation in (\ref{eq_b_cond_knots_half}), we may obtain an explicit expression of $b_{m,-\ell,k}$ by using Cramer's rule; that is,
\begin{equation}
b_{m,-\ell,k} = \frac{\det S_{m,\ell,k}}{\det S_m},
\label{eq_b_form_determinants}
\end{equation}
where
\begin{equation*}
S_m := \begin{bmatrix}
N_{{\ttt}_{-m+1},m}(a) & N_{{\ttt}_{-m+2},m}(a) &\cdots & N_{{\ttt}_0,m}(a) \\
N_{{\ttt}_{-m+1},m}^{\prime}(a) & N_{{\ttt}_{-m+2},m}^{\prime}(a) & \cdots & N_{{\ttt}_0,m}^{\prime}(a) \\
\vdots & \vdots & \ddots & \vdots \\
N_{{\ttt}_{-m+1},m}^{(m-1)}(a) & N_{{\ttt}_{-m+2},m}^{(m-1)}(a) & \cdots & N_{{\ttt}_0,m}^{(m-1)}(a) \\
\end{bmatrix},
\end{equation*}
and $S_{m,\ell,k}$ is obtained from $S_m$ by replacing its $(k+1)^{\textup{th}}$ column with the column vector ${\ddd}_{\ell} := [\underbrace{0,\ldots,0}_{\ell},1,\underbrace{0,\ldots,0}_{m-\ell-1}]^T;$ that is,
\begin{equation}
S_{m,\ell,k} := \begin{bmatrix}
N_{{\ttt}_{-m+1},m}(a) & \cdots & N_{{\ttt}_{-m+k},m}(a) & 0 & N_{{\ttt}_{-m+k+2},m}(a) & \cdots & N_{{\ttt}_0,m}(a) \\
\vdots & \ddots & \vdots & \vdots & \vdots & \ddots & \vdots \\
N_{{\ttt}_{-m+1},m}^{(\ell-1)}(a) & \cdots & N_{{\ttt}_{-m+k},m}^{(\ell-1)}(a) & 0 & N_{{\ttt}_{-m+k+2},m}^{(\ell-1)}(a) & \cdots & N_{{\ttt}_0,m}^{(\ell-1)}(a) \\
N_{{\ttt}_{-m+1},m}^{(\ell)}(a) & \cdots & N_{{\ttt}_{-m+k},m}^{(\ell)}(a) & 1 & N_{{\ttt}_{-m+k+2},m}^{(\ell)}(a) & \cdots & N_{{\ttt}_0,m}^{(\ell)}(a) \\
N_{{\ttt}_{-m+1},m}^{(\ell+1)}(a) & \cdots & N_{{\ttt}_{-m+k},m}^{(\ell+1)}(a) & 0 & N_{{\ttt}_{-m+k+2},m}^{(\ell+1)}(a) & \cdots & N_{{\ttt}_0,m}^{(\ell+1)}(a) \\
\vdots & \ddots & \vdots & \vdots & \vdots & \ddots & \vdots \\
N_{{\ttt}_{-m+1},m}^{(m-1)}(a) & \cdots & N_{{\ttt}_{-m+k},m}^{(m-1)}(a) & 0 & N_{{\ttt}_{-m+k+2},m}^{(m-1)}(a) & \cdots & N_{{\ttt}_0,m}^{(m-1)}(a)
\end{bmatrix}.
\label{eq_S_mlk}
\end{equation}
From the construction of the B-splines $N_{{\ttt}_{-m+1},m},\ldots,N_{{\ttt}_0,m}$, it is clear that $S_{m}$ is a lower triangular matrix, so that its determinant is simply given by
\begin{equation}
\det S_m = \prod_{j=0}^{m-1}N_{{\ttt}_{-m+1+j},m}^{(j)}(a).
\label{eq_det_S_m}
\end{equation}
To bound the determinant of $S_{m,\ell,k}$, we consider three cases. First, if $k<\ell$, then $S_{m,\ell,k}$ is also a lower triangular matrix with a zero appearing on the main diagonal in the $(k+1)^{\textup{th}}$ column, so that $\det S_{m,\ell,k} = 0.$ In this case, we therefore have, from (\ref{eq_b_form_determinants}), that $b_{m,-\ell,k} = 0$.\\

Next, suppose that $k=\ell$. In this case, $S_{m,\ell,k}$ in (\ref{eq_S_mlk}) is again a lower triangular matrix, with determinant
\begin{equation*}
\det S_{m,\ell,\ell} = \prod_{\substack{j=0\\j\neq\ell}}^{m-1} N_{{\ttt}_{-m+1+j},m}^{(j)}(a).
\end{equation*}
Therefore, from (\ref{eq_b_form_determinants}) and using (\ref{eq_det_S_m}), we have
\begin{equation}
b_{m,-\ell,\ell} = \frac{1}{N_{{\ttt}_{-m+1+\ell},m}^{(\ell)}(a)}.
\label{eq_proof_b_1}
\end{equation}
If $\ell=0$, it follows from the construction of ${\ttt}_{-m+1}$ in (\ref{eq_ttt_knots_half_m_even}), (\ref{eq_ttt_knots_half_m_odd}), that $N_{{\ttt}_{-m+1},m}(a) = 1,$ so that (\ref{eq_proof_b_1}) yields
\begin{equation}
b_{m,0,0} = 1.
\label{eq_proof_b_2}
\end{equation}
If $\ell\neq 0$, we may apply the first set of inequalities in (\ref{eq_bound_b_spline_der_scheme_H}) in Lemma \ref{lemma_bind_boundary_b_splines}, with $i=j=\ell$, in (\ref{eq_proof_b_1}) to deduce that
\begin{equation}
|b_{m,-\ell,\ell}| \leq \frac{\rho^{m-1}}{m-1}.
\label{eq_proof_b_3}
\end{equation}
Lastly, let $k>\ell$, so that $S_{m,\ell,k}$ in (\ref{eq_S_mlk}) is no longer a lower triangular matrix. In this case, since $S_{m,\ell,k}$ is an $m\times m$ matrix, and from the upper bound in (\ref{eq_bound_b_spline_der_scheme_H}) in Lemma \ref{lemma_bind_boundary_b_splines}, we deduce that
\begin{equation*}
|\det S_{m,\ell,k}| \leq m!\left(\frac{(m-1)!(m-1)!}{\lambda^{m-1}}\right)^m.
\end{equation*}
This, together with (\ref{eq_b_form_determinants}), (\ref{eq_det_S_m}) and the lower bound in (\ref{eq_bound_b_spline_der_scheme_H}), yields
\begin{equation}
|b_{m,-\ell,k}| \leq m!\left(\frac{\rho^{m-1} (m-1)!(m-2)!}{\lambda^{m-1}}\right)^m.
\label{eq_proof_b_4}
\end{equation}
The result follows by combining (\ref{eq_proof_b_2})-(\ref{eq_proof_b_4}) and the definition of $\tau$ in (\ref{eq_tau}). \end{proof}

Using Lemma \ref{lemma_upper_bound_b}, we may now obtain the following upper bounds on the spline molecules $L_{m,i}$:

\begin{theorem}
For an integer $m\geq3$, and with $\tau$ defined in (\ref{eq_tau}),
\begin{equation}
|L_{m,i}(x)| \leq
\left\{
\begin{aligned}
& m\tau, \quad i=-m+1,\ldots,0, \ N,\ldots,N+m-1;\\
& 1, \quad i=1,\ldots,N-1, \\
\end{aligned}
\right.
\label{eq_L_H_upper_bound}
\end{equation}
for all $x\in[a,b]$.
\label{thm_upper_bound_L}
\end{theorem}

\begin{proof}
First, if $i\in\left\{1,\ldots,N-1\right\}$, the result follows immediately from the construction of $L_{m,i}$ in (\ref{eq_L_defn_knots_half}). On the other hand, for $i\in\left\{-m+1,\ldots,0\right\}$ or $i\in\left\{N,\ldots,N+m-1\right\}$, we have, from the spline molecule definition (\ref{eq_L_defn_knots_half}), (\ref{eq_b_H_upper_bound}) in Lemma \ref{lemma_upper_bound_b}, the properties of the normalized B-splines (see, for example, \cite[Theorem~6.4]{chui1997wavelets}) and the construction of the knot sequences ${\ttt}_{-m+1},\ldots,{\ttt}_0$ in (\ref{eq_ttt_knots_half_m_even})-(\ref{eq_ttt_knots_half_m_odd}),
\begin{equation*}
|L_{m,i}(x)| \leq \sum_{k=0}^{m-1} |b_{m,i,k} N_{{\ttt}_{-m+1+k},m}(x)| \leq m\tau. \qedhere
\end{equation*}
\end{proof}

\subsection{Supremum norm approximation error of blending spline interpolation}
We are now in a position to analyze the approximation order of the blending interpolation operator ${\PP}_m$. In the following, $||\cdot||_{\infty,[x_i,x_{i+1}]}$ denotes the uniform (or supremum) norm on the interval $[x_i,x_{i+1}]$; that is $||g||_{\infty,[x_i,x_{i+1}]} := \sup \left\{ |g(x)|: \ x\in[x_i,x_{i+1}] \right\}$.

\begin{theorem}
Let $f\in C^m[a,b]$, and define
\begin{equation}
\varepsilon := \sup_{n=0,\ldots,N} |x_{n+1} - x_n|.
\label{eq_epsilon}
\end{equation}
Let $\gamma$, $\delta$ and $\tau$ be defined as in (\ref{eq_gamma_delta_H}) and (\ref{eq_tau}), respectively. Then the supremum norm approximation error of blending spline interpolation is given by
\begin{equation}
|| f - {\PP}_m f ||_{\infty,[x_i,x_{i+1}]} 
\leq \left\{
\begin{aligned}
& ||f^{(m)}||_{\infty,[x_i,x_{i+1}]} U_{m,\varepsilon,\gamma,\delta,\tau}, \quad i=0,\ 1;\\
&  ||f^{(m)}||_{\infty,[x_i,x_{i+1}]} V_{m,\varepsilon,\gamma,\delta}, \quad i=2,\ldots,N-m+1;\\
&  ||f^{(m)}||_{\infty,[x_i,x_{i+1}]} W_{m,\varepsilon,\gamma,\delta}, \quad i=N-m+2,\ldots,N-2;\\
&  ||f^{(m)}||_{\infty,[x_i,x_{i+1}]} X_{m,\varepsilon,\gamma,\delta,\tau}, \quad i=N-1, \ N,\\
\end{aligned}
\right.
\label{eq_error_P_H}
\end{equation}
where
\begin{equation*}
\left\{
\begin{aligned}
& U_{m,\varepsilon,\gamma,\delta,\tau} := \varepsilon^{m}\left(A_1 + A_2\left( \tfrac{\gamma}{\delta} \right)^{m-1} + A_3\tau\left( \tfrac{\gamma}{\delta} \right)^{m-1}\left(\tfrac{2}{\delta}\right)^{m-1}\right) \\
& \quad \quad \quad \quad \quad \quad \quad \quad{+\varepsilon \left(A_4\left( \tfrac{\gamma}{\delta} \right)^{m-1} + A_5\tau\left( \tfrac{\gamma}{\delta} \right)^{m-1}\left(\tfrac{2}{\delta}\right)^{m-1} + A_6\tau\right); } \\
& V_{m,\varepsilon,\gamma,\delta} := \varepsilon^{m}\left(B_1 + B_2\left( \tfrac{\gamma}{\delta} \right)^{m-1}\right);\\
& W_{m,\varepsilon,\gamma,\delta} := \varepsilon^{m}\left(C_1 + C_2\left( \tfrac{\gamma}{\delta} \right)^{m-1}\right) + \varepsilon C_3 \left( \tfrac{\gamma}{\delta} \right)^{m-1}\left(\tfrac{1-\varepsilon^{m-1}}{1-\varepsilon}\right);\\
& X_{m,\varepsilon,\gamma,\delta,\tau} := \varepsilon^{m}\left(D_1 + D_2\left( \tfrac{\gamma}{\delta} \right)^{m-1} + D_3\tau\left( \tfrac{\gamma}{\delta} \right)^{m-1}\left(\tfrac{2}{\delta}\right)^{m-1}\right) \\
& \pushright{+ \varepsilon\left(D_4\left( \tfrac{\gamma}{\delta} \right)^{m-1}\left(\tfrac{1-\varepsilon^{m-1}}{1-\varepsilon}\right) + D_5\tau\left(\tfrac{1-\varepsilon^{m-1}}{1-\varepsilon}\right)\right. }\\
& \pushright{\left. + D_6\tau\left( \tfrac{\gamma}{\delta} \right)^{m-1}\left(\tfrac{2}{\delta}\right)^{m-1}\left(\tfrac{1-\varepsilon^{m-1}}{1-\varepsilon}\right) \right), }\\
\end{aligned}
\right.
\end{equation*}
and where the $A$'s, $B$'s, $C$'s and $D$'s are constants depending only on $m$.
\label{thm_error_blending_interpolation}
\end{theorem}

\begin{proof}
Let $i\in\left\{ 0,\ldots, N \right\}$ be fixed, and let $x\in\left[x_i, x_{i+1}\right]$. Then, with the definition
\begin{equation}
h(x,y) := (x-y)_{+}^{m-1},
\label{eq_h_def}
\end{equation}
it follows that
\begin{equation}
f(x) - ({\PP}_m f)(x) = \int_{x_i}^{x_{i+1}} \frac{f^{(m)}(y)}{(m-1)!}\left[h(x,y) - ({\PP}_m h(\cdot,y))(x) \right] \mathrm{d}y,
\label{eq_proof_error_4}
\end{equation}
since ${\PP}_m$ preserves polynomials in $\pi_{m-1}$.\\

We start by considering the second inequality in (\ref{eq_error_P_H}). To this end, let \\$i\in\left\{m-1,\ldots,N-m+1 \right\}$, with $x\in[x_i,x_{i+1}]$. In the following, we suppress the variable of integration $y$, so that $h(x,y) = h(x)$. From the definition of ${\PP}_m$ in (\ref{eq_PP_knots_half}), the support properties (\ref{eq_M_support_knots_half}) and (\ref{eq_L_support_knots_half}) of the spline molecules, and the definition of ${\QQ}_m$ in (\ref{eq_QQ_knots_half}), we have
\begin{multline}
h(x) - ({\PP}_mh)(x)
= h(x) - \sum_{j=i+1-m}^{i+m-1} h(y_j)M_{m,j}(x) - \sum_{j=i-1}^{i+1}h(y_j)L_{m,j}(x) \\
+ \sum_{j=i-1}^{i+1} \left[ \sum_{k=j+1-m}^{j+m-1} h(y_k) M_{m,k}(y_j) \right]L_{m,j}(x).
\label{eq_proof_1}
\end{multline}
Next, we observe that
\begin{equation*}
h(y_j) = (y_j-y)^{m-1}_{+} = 0, \quad y\geq y_j,
\end{equation*}
from the definition of $h$ in (\ref{eq_h_def}), so that, since $y\in [x_i, x_{i+1}]$ and $x_j < y_j < x_{j+1}$ for all $j=1,\ldots,N-1$ (from (\ref{eq_yyy_knots_half})),
\begin{equation}
h(y_j) = 0, \quad j \leq i-1.
\label{eq_h_zero}
\end{equation}
Therefore, we may use the upper bounds (\ref{eq_M_H_upper_bound}) in Theorem \ref{thm_upper_bound_M} and (\ref{eq_L_H_upper_bound}) in Theorem \ref{thm_upper_bound_L} to deduce that
\begin{align*}
& |h(x) - ({\PP}_mh)(x) | \\
& \leq |h(x)| + \frac{1}{(m-2)!}\left(\frac{\gamma}{\delta}\right)^{m-1} \sum_{j=i}^{i+m-1}|h(y_j)| + \sum_{j=i}^{i+1} |h(y_j)|  { + \frac{1}{(m-2)!}\left(\frac{\gamma}{\delta}\right)^{m-1}\sum_{j=i-1}^{i+1}\sum_{k=j+1-m}^{j+m-1} |h(y_k)| } \\
& \leq \varepsilon^{m-1} + \frac{1}{(m-2)!}\left(\frac{\gamma}{\delta}\right)^{m-1} \left(\varepsilon^{m-1} + (2\varepsilon)^{m-1} + \cdots + (m\varepsilon)^{m-1}\right) + \left(\varepsilon^{m-1} + (2\varepsilon)^{m-1}\right) \\
& \quad \quad {+ \frac{1}{(m-2)!}\left(\frac{\gamma}{\delta}\right)^{m-1} \left[ \sum_{k=i}^{i+m-2} |h(y_k)| + \sum_{k=i}^{i+m-1} |h(y_k)| + \sum_{k=i}^{i+m} |h(y_k)| \right] } \\
& \leq (2+2^{m-1})\varepsilon^{m-1} + \frac{1}{(m-2)!}\left(\frac{\gamma}{\delta}\right)^{m-1} \left[\left(\varepsilon^{m-1} + \cdots + ((m-1)\varepsilon)^{m-1} \right) \right. \\
& \quad \quad { \left. + 2\left(\varepsilon^{m-1}+\cdots+(m\varepsilon)^{m-1}\right) + \left(\varepsilon^{m-1}+\cdots+((m+1)\varepsilon)^{m-1}\right) \right] } \\
& =(2+2^{m-1})\varepsilon^{m-1} + \tilde{B}_1\varepsilon^{m-1}\left(\frac{\gamma}{\delta}\right)^{m-1},
\end{align*}
where $\tilde{B}_1$ is a constant only depending on $m$. Therefore, (\ref{eq_proof_error_4}) becomes
\begin{align*}
|f(x) - ({\PP}_m f)(x)|
\leq ||f^{(m)}||_{\infty,[x_i,x_{i+1}]}\left[ \frac{(2+2^{m-1})\varepsilon^{m}}{(m-1)!} + \frac{\tilde{B}_1\varepsilon^{m}}{(m-1)!}\left(\frac{\gamma}{\delta}\right)^{m-1} \right],
\end{align*}
from which our result follows, with $B_1=\frac{2+2^{m-1}}{(m-1)!}$ and $B_2=\frac{\tilde{B}_1}{(m-1)!}$.\\

Next, for a fixed $i\in\left\{2,\ldots,m-2\right\}$, let $x\in [x_i,x_{i+1}]$. By applying the definition of ${\PP}_m$ in (\ref{eq_PP_knots_half}), the support properties (\ref{eq_M_support_knots_half}) and (\ref{eq_L_support_knots_half}) of the spline molecules, and the definition of ${\QQ}_m$ in (\ref{eq_QQ_knots_half}), we deduce that
\begin{align}
& h(x) - ({\PP}_mh)(x)\nonumber \\
&= h(x) - \sum_{j=i+1-m}^{-1} h^{(-j)}(a) M_{m,j}(x) - \sum_{j=0}^{i+m-1} h(y_j)M_{m,j}(x) - \sum_{j=i-1}^{i+1}h(y_j)L_{m,j}(x) \nonumber\\
& \quad \quad{ + \sum_{j=i-1}^{i+1} \left[ \sum_{k=j+1-m}^{-1} h^{(-k)}(a)M_{m,k}(y_j) + \sum_{k=0}^{j+m-1} h(y_k) M_{m,k}(y_j) \right]L_{m,j}(x). }
\label{eq_proof_2}
\end{align}
Now, we deduce from the definition of $h$ in (\ref{eq_h_def}) that
\begin{equation}
h^{(n)}(x) = (m-1)(m-2)\cdots(m-n)(x-y)^{m-1-n}_{+}, \quad n=0,\ldots,m-1,
\label{eq_h_der_1}
\end{equation}
so that
\begin{equation}
h^{(n)}(a) = \begin{cases}
(m-1)!, & n=m-1; \\
0, & n=0,\ldots,m-2.
\end{cases}
\label{eq_h_der_2}
\end{equation}
Therefore, (\ref{eq_proof_2}) yields the same formulation as in (\ref{eq_proof_1}), so that (\ref{eq_proof_error_4}) yields the same result as before.\\

Next, let $x\in[x_i,x_{i+1}], \ i\in\left\{0,1\right\}$. We proceed to derive the first inequality in (\ref{eq_error_P_H}). In this case, we have
\begin{align*}
& h(x) - ({\PP}_mh)(x) \\
& = h(x) - \sum_{j=i+1-m}^{-1} h^{(-j)}(a) M_{m,j}(x) - \sum_{j=0}^{i+m-1}h(y_j) M_{m,j}(x) \\
& \quad \quad{ - \sum_{j=1-m}^{-1}h^{(-j)}(a)L_{m,j}(x) - \sum_{j=0}^{i+1} h(y_j) L_{m,j}(x) } \\
& \quad \quad{ + \sum_{j=1-m}^{-1} \left[ \sum_{k=j+1-m}^{-1} h^{(-k)}(a)M_{m,k}^{(-j)}(a) + \sum_{k=0}^{j+m-1} h(y_k) M_{m,k}^{(-j)}(a) \right] L_{m,j}(x) } \\
& \quad \quad{ + \sum_{j=0}^{i+1} \left[ \sum_{k=j+1-m}^{-1} h^{(-k)}(a)M_{m,k}(y_j) + \sum_{k=0}^{j+m-1} h(y_k) M_{m,k}(y_j) \right]L_{m,j}(x), }
\end{align*}
from the definition of ${\PP}_m$ in (\ref{eq_PP_knots_half}), the support properties (\ref{eq_M_support_knots_half}) and (\ref{eq_L_support_knots_half}) of the spline molecules, and the definition of ${\QQ}_m$ in (\ref{eq_QQ_knots_half}). Now, we use the upper bounds (\ref{eq_M_H_upper_bound}) in Theorem \ref{thm_upper_bound_M} and (\ref{eq_L_H_upper_bound}) in Theorem \ref{thm_upper_bound_L} and (\ref{eq_M_H_der_upper_bound}) in Theorem \ref{thm_upper_bound_M_der}, together with (\ref{eq_h_zero}) and (\ref{eq_h_der_2}), to obtain
\begin{align*}
& |h(x) - ({\PP}_mh)(x)| \\
& \leq |h(x)| + \frac{1}{(m-2)!}\left(\frac{\gamma}{\delta}\right)^{m-1}\sum_{j=1}^{m-1-i}|h^{(j)}(a)| + \frac{1}{(m-2)!}\left(\frac{\gamma}{\delta}\right)^{m-1}\sum_{j=1}^{i+m-1}|h(y_j)| \\
& \quad \quad{ + m\tau\sum_{j=1}^{m-1}|h^{(j)}(a)| + m\sum_{j=1}^{i+1}|h(y_j)| } \\
& \quad \quad{ + m\tau\frac{1}{(m-2)!}\left(\frac{\gamma}{\delta}\right)^{m-1}m\left(\frac{2}{\delta}\right)^{m-1} \sum_{j=1}^{m-1} \left[ \sum_{k=1}^{m-1}|h^{(k)}(a)| + \sum_{k=1}^{m-1-j}|h(y_k)| \right] } \\
& \quad \quad{ + m\frac{1}{(m-2)!}\left(\frac{\gamma}{\delta}\right)^{m-1} \sum_{j=0}^{i+1} \left[ \sum_{k=1}^{m-1-j}|h^{(k)}(a)| + \sum_{k=1}^{j+m-1}|h(y_k)| \right] } \\
& \leq \varepsilon^{m-1} + \frac{1}{(m-2)!}\left(\frac{\gamma}{\delta}\right)^{m-1}(m-1)! { + \frac{1}{(m-2)!}\left(\frac{\gamma}{\delta}\right)^{m-1} \left( (2\varepsilon)^{m-1} + \cdots + ((m+1)\varepsilon)^{m-1} \right) } \\
& \quad \quad{ + m\tau(m-1)! + m\left( (2\varepsilon)^{m-1} + (3\varepsilon)^{m-1}\right) }\\
& \quad \quad{ + \frac{m\tau}{(m-2)!}\left(\frac{\gamma}{\delta}\right)^{m-1}m\left(\frac{2}{\delta}\right)^{m-1}   \left[ (m-1)(m-1)! + \left((2\varepsilon)^{m-1} + \cdots + ((m-1)\varepsilon)^{m-1}\right) \right] } \\
& \quad \quad{ + \frac{m}{(m-2)!}\left(\frac{\gamma}{\delta}\right)^{m-1} \left[ (m-1)! + 3\left( (2\varepsilon)^{m-1} + \cdots + ((m+2)\varepsilon)^{m-1} \right) \right] } \\
& \leq (1+m2^{m-1}+m3^{m-1})\varepsilon^{m-1} + \frac{(m+1)(m-1)!}{(m-2)!}\left(\frac{\gamma}{\delta}\right)^{m-1} + \tilde{A}_1\varepsilon^{m-1}\left(\frac{\gamma}{\delta}\right)^{m-1}\\
& \quad \quad + m\tau(m-1)! {  + \frac{m^2\tau(m-1)(m-1)!}{(m-2)!}\left(\frac{\gamma}{\delta}\right)^{m-1}\left(\frac{2}{\delta}\right)^{m-1} } {+ \tilde{A}_2\tau\varepsilon^{m-1}\left(\frac{\gamma}{\delta}\right)^{m-1}\left(\frac{2}{\delta}\right)^{m-1},}
\end{align*}
where $\tilde{A}_1$ and $\tilde{A}_2$ are constants only depending on $m$. Therefore, (\ref{eq_proof_error_4}) yields
\begin{align*}
& |f(x) - ({\PP}_m f)(x)| \\
& \leq ||f^{(m)}||_{\infty,[x_i,x_{i+1}]}\left[ \frac{(1+m2^{m-1}+m3^{m-1})\varepsilon^{m}}{(m-1)!} + \frac{(m+1)\varepsilon}{(m-2)!}\left(\frac{\gamma}{\delta}\right)^{m-1}  + \frac{\tilde{A}_1\varepsilon^{m}}{(m-1)!}\left(\frac{\gamma}{\delta}\right)^{m-1} \right.\\
& \quad \quad + m\tau\varepsilon { \left. + \frac{m^2\tau\varepsilon(m-1)}{(m-2)!}\left(\frac{\gamma}{\delta}\right)^{m-1}\left(\frac{2}{\delta}\right)^{m-1} + \frac{\tilde{A}_2\tau\varepsilon^{m}}{(m-1)!}\left(\frac{\gamma}{\delta}\right)^{m-1}\left(\frac{2}{\delta}\right)^{m-1} \right],  }
\end{align*}
from which our result follows with $A_1=\frac{(1+m2^{m-1}+m3^{m-1})}{(m-1)!}$, $A_2=\frac{\tilde{A}_1}{(m-1)!}$, $A_3=\frac{\tilde{A}_2}{(m-1)!}$, $A_4=\frac{(m+1)}{(m-2)!}$, $A_5=\frac{m^2(m-1)}{(m-2)!}$ and $A_6=m$.\\

Next, we turn our attention to the third inequality in (\ref{eq_error_P_H}), and let \\$i\in\left\{N-m+2,\ldots,N-2\right\}$, with $x\in[x_i,x_{i+1}]$. In this case, we have, from the definition of ${\PP}_m$ in (\ref{eq_PP_knots_half}), the spline molecule support properties in (\ref{eq_M_support_knots_half}) and (\ref{eq_L_support_knots_half}), and the definition of ${\QQ}_m$ in (\ref{eq_QQ_knots_half}),
\begin{align*}
& h(x) - ({\PP}_mh)(x) \\
& = h(x) - \sum_{j=i+1-m}^{N}h(y_j)M_{m,j}(x) - \sum_{j=N+1}^{i+m-1}h^{(j-N)}(b) M_{m,j}(x)  -\sum_{j=i-1}^{i+1} h(y_j)L_{m,j}(x) \\
& \quad \quad + \sum_{j=i-1}^{i+1} \left[ \sum_{k=j+1-m}^{N} h(y_k) M_{m,k}(y_j) + \sum_{k=N+1}^{j+m-1} h^{(k-N)}(b) M_{m,k}(y_j) \right] L_{m,j}(x).
\end{align*}
Now, from (\ref{eq_h_def}) and (\ref{eq_h_der_1}), we have
\begin{align}
& \sum_{j=1}^{i+m-1-N}|h^{(j)}(b)| \\
& = \sum_{j=1}^{i+m-1-N} |(m-1)(m-2)\cdots(m-j)(b-y)^{m-1-j}_{+}| \nonumber \\
& \leq (N+1-i)\left[((m-1)\varepsilon)^{m-2} + ((m-1)(m-2)\varepsilon)^{m-3} + \cdots \right. \nonumber \\
& \quad \quad {\left. + ((m-1)(m-2)\cdots(m-(i+m-1-N))\varepsilon)^{m-1-(i+m-1-N)}\right]} \nonumber \\
& \leq (m-1)\sum_{k=1}^{m-1}\left[\prod_{\ell=1}^{k}(m-\ell)\right]^{m-1-k}\varepsilon^{m-1-k} \leq 
(m-1)((m-1)!)^{m-2}\frac{1-\varepsilon^{m-1}}{1-\varepsilon}.
\label{eq_h_der_3}
\end{align}
This, together with the upper bounds (\ref{eq_M_H_upper_bound}) and (\ref{eq_L_H_upper_bound}), and (\ref{eq_h_zero}), (\ref{eq_h_der_1}), leads to
\begin{align*}
& |h(x) - ({\PP}_mh)(x)| \\
& \leq |h(x)| + \frac{1}{(m-2)!}\left(\frac{\gamma}{\delta}\right)^{m-1}\sum_{j=i}^{N}|h(y_j)| + \frac{1}{(m-2)!}\left(\frac{\gamma}{\delta}\right)^{m-1}\sum_{j=1}^{i+m-1-N}|h^{(j)}(b)| \\
& \quad \quad + \sum_{j=i}^{i+1}|h(y_j)| + \frac{1}{(m-2)!}\left(\frac{\gamma}{\delta}\right)^{m-1}\sum_{j=i-1}^{i+1} \left[ \sum_{k=j+1-m}^{N}|h(y_k)| + \sum_{k=1}^{j+m-1-N}|h^{(k)}(b)| \right]\\
& \leq \varepsilon^{m-1} + \frac{1}{(m-2)!}\left(\frac{\gamma}{\delta}\right)^{m-1}\left( \varepsilon^{m-1} + \cdots + ((m-1)\varepsilon)^{m-1} \right) \\
& \quad \quad  + \frac{1}{(m-2)!}\left(\frac{\gamma}{\delta}\right)^{m-1} \left[ (m-1)((m-1)!)^{m-2} \frac{1-\varepsilon^{m-1}}{1-\varepsilon} \right] + \left(\varepsilon^{m-1} + (2\varepsilon)^{m-1}\right) \\
& \quad \quad  + \frac{1}{(m-2)!}\left(\frac{\gamma}{\delta}\right)^{m-1} \left[ 3\sum_{k=i}^{N}|h(y_k)| + \sum_{k=1}^{i+m-2-N}|h^{(k)}(b)| \right. \\
& \pushright{ \left. + \sum_{k=1}^{i+m-1-N}|h^{(k)}(b)| + \sum_{k=1}^{i+m-N}|h^{(k)}(b)| \right] } \\
& \leq (2+2^{m-1})\varepsilon^{m-1} + \tilde{C}_1\varepsilon^{m-1}\left(\frac{\gamma}{\delta}\right)^{m-1} + \frac{4(m-1)((m-1)!)^{m-2}}{(m-2)!}\left(\frac{\gamma}{\delta}\right)^{m-1}  \frac{1-\varepsilon^{m-1}}{1-\varepsilon},
\end{align*}
where $\tilde{C}_1$ is a constant only depending on $m$. Therefore, (\ref{eq_proof_error_4}) becomes
\begin{multline*}
|f(x) - ({\PP}_m f)(x)|
\leq ||f^{(m)}||_{\infty,[x_i,x_{i+1}]}\left[ \frac{(2+2^{m-1})\varepsilon^{m}}{(m-1)!} + \frac{\tilde{C}_1\varepsilon^{m}}{(m-1)!}\left(\frac{\gamma}{\delta}\right)^{m-1} \right. \\
\left. + \frac{4(m-1)((m-1)!)^{m-3}\varepsilon}{(m-2)!}\left(\frac{\gamma}{\delta}\right)^{m-1}  \left( \frac{1-\varepsilon^{m-1}}{1-\varepsilon} \right) \right],
\end{multline*}
and our result follows with $C_1=\frac{(2+2^{m-1})}{(m-1)!}$, $C_2=\frac{\tilde{C}_1}{(m-1)!}$ and $C_3=\frac{4(m-1)((m-1)!)^{m-3}}{(m-2)!}$.\\

Lastly, let $i\in\left\{N-1,N\right\}$, with $x\in[x_i,x_{i+1}]$. We proceed to prove the fourth inequality in (\ref{eq_error_P_H}). In this case, we have, from the definition of ${\PP}_m$ in (\ref{eq_PP_knots_half}), the spline molecule support properties in (\ref{eq_M_support_knots_half}) and (\ref{eq_L_support_knots_half}), and the definition of ${\QQ}_m$ in (\ref{eq_QQ_knots_half}),
\begin{align*}
& h(x) - ({\PP}_mh)(x) \\
& = h(x) - \sum_{j=i+1-m}^{N} h(y_j) M_{m,j}(x) - \sum_{j=N+1}^{i+m-1} h^{(j-N)}(b)M_{m,j}(x) \\
& \quad \quad - \sum_{j=i-1}^{N}h(y_j)L_{m,j}(x) - \sum_{j=N+1}^{N+m-1} h^{(j-N)}(b) L_{m,j}(x) \\
& \quad \quad + \sum_{j=i-1}^{N} \left[ \sum_{k=j+1-m}^{N} h(y_k) M_{m,k}(y_j) + \sum_{k=N+1}^{j+m-1} h^{(k-N)}(b)M_{m,k}(y_j) \right] L_{m,j}(x) \\
& \quad \quad + \sum_{j=N+1}^{N+m-1} \left[ \sum_{k=j+1-m}^{N} h(y_k) M_{m,k}^{(j-N)}(b) + \sum_{k=N+1}^{j+m-1} h^{(k-N)}(b)M_{m,k}^{(j-N)}(b) \right] L_{m,j}(x).
\end{align*}
Next, we apply the upper bounds (\ref{eq_M_H_upper_bound}), (\ref{eq_L_H_upper_bound}) and (\ref{eq_M_H_der_upper_bound}), together with (\ref{eq_h_zero}), (\ref{eq_h_der_1}) and (\ref{eq_h_der_3}), to deduce that
\begin{align*}
& |h(x) - ({\PP}_mh)(x)| \\
& \leq |h(x)| + \frac{1}{(m-2)!}\left(\frac{\gamma}{\delta}\right)^{m-1} \sum_{j=i}^{N} |h(y_j)| + \frac{1}{(m-2)!}\left(\frac{\gamma}{\delta}\right)^{m-1} \sum_{j=1}^{i+m-1-N} |h^{(j)}(b)| \\
& \quad \quad + m\sum_{j=i}^{N}|h(y_j)| + m\tau\sum_{j=1}^{m-1}|h^{(j)}(b)| \\
& \quad \quad + m\frac{1}{(m-2)!}\left(\frac{\gamma}{\delta}\right)^{m-1}\sum_{j=i-1}^{N} \left[ \sum_{k=j+1-m}^{N} |h(y_k)| + \sum_{k=1}^{j+m-1-N}|h^{(k)}(b)| \right] \\
& \quad \quad + m\tau\frac{1}{(m-2)!}\left(\frac{\gamma}{\delta}\right)^{m-1}m\left(\frac{2}{\delta}\right)^{m-1} { \sum_{j=1}^{m-1} \left[ \sum_{k=j+1-m+N}^{N}|h(y_k)| + \sum_{k=1}^{m-1} |h^{(k)}(b)| \right] }\\
& \leq \varepsilon^{m-1} + \frac{1}{(m-2)!}\left(\frac{\gamma}{\delta}\right)^{m-1} \left(\varepsilon^{m-1} + (2\varepsilon)^{m-1}\right) \\
& \quad \quad + \frac{1}{(m-2)!}\left(\frac{\gamma}{\delta}\right)^{m-1} 2((m-1)!)^{m-2}\frac{1-\varepsilon^{m-1}}{1-\varepsilon} + m\left( \varepsilon^{m-1} + (2\varepsilon)^{m-1} \right)\\
& \quad \quad  + m\tau2((m-1)!)^{m-2}\frac{1-\varepsilon^{m-1}}{1-\varepsilon}\\
& \quad \quad + \frac{m}{(m-2)!}\left(\frac{\gamma}{\delta}\right)^{m-1} 3\left[ \left(\varepsilon^{m-1} + (2\varepsilon)^{m-1}\right) + 2((m-1)!)^{m-2}\frac{1-\varepsilon^{m-1}}{1-\varepsilon} \right] \\
& \quad \quad + \frac{m\tau}{(m-2)!}\left(\frac{\gamma}{\delta}\right)^{m-1}m\left(\frac{2}{\delta}\right)^{m-1} (m-1) \times \\
& \pushright{  \left[ \left( \varepsilon^{m-1} + (2\varepsilon)^{m-1} \right) + 2((m-1)!)^{m-2}\frac{1-\varepsilon^{m-1}}{1-\varepsilon} \right] }\\
& \leq (1+m+m2^{m-1})\varepsilon^{m-1} + \frac{(1+2^{m-1}+3m+3m2^{m-1})\varepsilon^{m-1}}{(m-2)!}\left(\frac{\gamma}{\delta}\right)^{m-1} \\
& \quad \quad + \frac{(6m+2)((m-1)!)^{m-2}}{(m-2)!}\left(\frac{\gamma}{\delta}\right)^{m-1} \frac{1-\varepsilon^{m-1}}{1-\varepsilon}  \\
& \quad \quad + 2m\tau((m-1)!)^{m-2}\frac{1-\varepsilon^{m-1}}{1-\varepsilon} + \frac{(1+2^{m-1})m^2\tau\varepsilon^{m-1}(m-1)}{(m-2)!}\left(\frac{\gamma}{\delta}\right)^{m-1}\left(\frac{2}{\delta}\right)^{m-1} \\
& \pushright{ + \frac{2m^2\tau(m-1)((m-1)!)^{m-2}}{(m-2)!}\left(\frac{\gamma}{\delta}\right)^{m-1}\left(\frac{2}{\delta}\right)^{m-1}\frac{1-\varepsilon^{m-1}}{1-\varepsilon}. }
\end{align*}
Therefore, (\ref{eq_proof_error_4}) becomes
\begin{align*}
&|f(x) - ({\PP}_m f)(x)| \\
& \leq ||f^{(m)}||_{\infty,[x_i,x_{i+1}]}\left[ \frac{(1+m+m2^{m-1})\varepsilon^{m}}{(m-1)!} + \frac{(1+2^{m-1}+3m+3m2^{m-1})\varepsilon^{m}}{(m-2)!(m-1)!}\left(\frac{\gamma}{\delta}\right)^{m-1} \right. \\
& \quad \quad \left. + \frac{(6m+2)((m-1)!)^{m-3}\varepsilon}{(m-2)!}\left(\frac{\gamma}{\delta}\right)^{m-1} \left( \frac{1-\varepsilon^{m-1}}{1-\varepsilon}\right) + 2m((m-1)!)^{m-3}\tau\varepsilon\left( \frac{1-\varepsilon^{m-1}}{1-\varepsilon}\right) \right. \\
& \quad \quad \left. + \frac{(1+2^{m-1})m^2\tau\varepsilon^{m}(m-1)}{(m-1)!(m-2)!}\left(\frac{\gamma}{\delta}\right)^{m-1}\left(\frac{2}{\delta}\right)^{m-1} \right. \\
& \pushright{ \left. + \frac{2m^2\tau\varepsilon(m-1)((m-1)!)^{m-3}}{(m-2)!}\left(\frac{\gamma}{\delta}\right)^{m-1}\left(\frac{2}{\delta}\right)^{m-1}\left( \frac{1-\varepsilon^{m-1}}{1-\varepsilon}\right) \right], }
\end{align*}
thereby completing our proof of (\ref{eq_error_P_H}), with $D_1=\frac{(1+m+m2^{m-1})}{(m-1)!}$, $D_2=\frac{1+2^{m-1}+3m+3m2^{m-1}}{(m-2)!(m-1)!}$, $D_3=\frac{(1+2^{m-1})m^2(m-1)}{(m-1)!(m-2)!}$, $D_4=\frac{(6m+2)((m-1)!)^{m-3}}{(m-2)!}$, $D_5=2m((m-1)!)^{m-3}$ and \\$D_6=\frac{2m^2(m-1)((m-1)!)^{m-3}}{(m-2)!}$.
\end{proof}

\section{Final remarks}
\label{sec_final}

In this paper, we developed a local polynomial spline interpolation scheme for arbitrary spline order on bounded intervals. Our method's local formulation, effective boundary considerations and interpolation error rate make it particularly useful for real-time implementation in real-world applications.

\section*{Acknowledgment}
This paper was written during my PhD studies at the University of Missouri-St. Louis. I would like to thank my mentor, Charles K. Chui, for his guidance during this research.

\bibliographystyle{plain}
\bibliography{mybib}

\end{document}